\documentclass[11pt]{amsart}
\usepackage{amsfonts, amsthm, amsmath, amssymb}
\usepackage[all]{xy}
\usepackage[top=1in,bottom=1in,left=1in,right=1in]{geometry}
\usepackage[pdftex]{hyperref}

\newcommand{\bB}{\mathbb{B}}
\newcommand{\bC}{\mathbb{C}}

\newcommand{\bM}{\mathbb{M}}

\newcommand{\bQ}{\mathbb{Q}}
\newcommand{\bR}{\mathbb{R}}

\newcommand{\bZ}{\mathbb{Z}}
\newcommand{\unit}{\mathbf{1}}

\newcommand{\cC}{\mathcal{C}}

\newcommand{\cH}{\mathcal{H}}

\newcommand{\cM}{\mathcal{M}}

\newcommand{\fg}{\mathfrak{g}}

\newcommand{\fh}{\mathfrak{h}}
\newcommand{\fH}{\mathfrak{H}}
\newcommand{\fm}{\mathfrak{m}}

\newcommand{\ad}{\operatorname{ad}}
\newcommand{\Aut}{\operatorname{Aut}}

\newcommand{\End}{\operatorname{End}}

\newcommand{\Hom}{\operatorname{Hom}}

\newcommand{\Ker}{\operatorname{Ker}}

\newcommand{\Rep}{\operatorname{Rep}}

\newcommand{\Tr}{\operatorname{Tr}}

\newtheorem{thm}{Theorem}[subsection]
\newtheorem*{conj}{Conjecture}
\newtheorem{prop}[thm]{Proposition}

\newtheorem{lem}[thm]{Lemma}
\newtheorem{cor}[thm]{Corollary}

\theoremstyle{definition}
\newtheorem{defn}[thm]{Definition}

\theoremstyle{remark}
\newtheorem{rem}[thm]{Remark}

\begin{document}

\title{Generalized Moonshine IV: Monstrous Lie algebras}
\author{Scott Huai-Lei Carnahan}
\address{Division of Mathematics, University of Tsukuba}
\email{carnahan@math.tsukuba.ac.jp}

\begin{abstract}
For each element of the Fischer-Griess Monster sporadic simple group, we construct an infinite dimensional Lie algebra equipped with a projective action of the centralizer of that element.  Our construction is given by a string-theoretic ``add a spacetime torus and quantize'' functor applied to an abelian intertwining algebra that is formed from a family of twisted modules of the Monster vertex operator algebra.  We prove that for all Fricke elements in the Monster, the characters of centralizers acting on the corresponding irreducible twisted modules are Hauptmoduln.  From these results, we resolve Norton's Generalized Moonshine Conjecture.
\end{abstract}

\maketitle

\tableofcontents

\section{Introduction}

For each element $g$ of the Monster, we construct an infinite dimensional Lie algebra equipped with a canonical projective action of the centralizer of $g$.  For those $g$ satisfying the Fricke property, we show that the Lie algebras are Borcherds-Kac-Moody with a real simple root, and we use the resulting twisted denominator identities to prove Norton's Generalized Moonshine Conjecture for the class of irreducible twisted $V^\natural$-modules.

\subsection{Generalized Moonshine}

We give a brief reminder of monstrous moonshine before describing its generalization.  In the late 1970s, McKay observed that the $q^1$-term in the $q$-expansion of the normalized modular function $J = q^{-1} + 196884q + 21493760q^2 + \cdots$ is one more than the dimension of the smallest nontrivial complex representation of the Fischer-Griess Monster simple group $\bM$, which is a sporadic simple group of order about $8 \times 10^{53}$.  This observation suggested a relationship between the $J$-function and representations of the Monster that was further reinforced by additional computations by McKay and Thompson. Conway and Norton then did a massive computation with characters, producing evidence that there is an infinite dimensional graded representation $V = \bigoplus_{n=0}^\infty V_n$ of the Monster, such that the graded trace $T_g(\tau) = \sum_{n=0}^\infty \Tr(g|V_n) q^{n-1}$ of any element is the Fourier expansion (under $q = e^{2\pi i \tau}$) of a holomorphic function on the complex upper half-plane $\fH$ that is moreover a Hauptmodul \cite{CN79}.  That is, for each $g \in \bM$, there is a group $\Gamma_g \subset SL_2(\bR)$ under which $T_g(\tau)$ is invariant, such that the quotient curve $X_g = \fH/\Gamma_g$ is genus zero, and $T_g$ generates the field of meromorphic functions on $X_g$.  A candidate representation $V^\natural$ was constructed by Frenkel, Lepowsky, and Meurman, who showed that it is a vertex operator algebra whose automorphism group is precisely $\bM$, and its graded dimension $T_1$ is $J$ \cite{FLM88}.  Borcherds then proved Conway and Norton's conjecture for $V^\natural$ \cite{B92}.

In the Conway-Norton paper, the authors suggest that there may be moonshine-like phenomena for groups other than $\bM$, and shortly afterward, Queen did some computations that yielded concrete evidence that one can obtain Hauptmoduln from characters of a variety of large finite groups \cite{Q81}.  In particular, the low-order coefficients of many such modular functions can be assembled from representations of sporadic simple groups other than the Monster.

For illumination, we consider a level 2 example.  The irreducible representations of the Baby Monster sporadic simple group $\bB$ have dimension $1, 4371, 96255, \ldots$, while the power series $q^{-1} + 4372q + 96256q^2 + \cdots$ is the $q$-expansion of $\frac{\Delta(\tau)}{\Delta(2\tau)} + \frac{\Delta(2\tau)}{\Delta(\tau)} + 24$, which is a Hauptmodul for $\Gamma_0(2)^+$.  It is difficult to resist the temptation to decompose the coefficients as small combinations of dimensions of irreducible representations: $4372 = 4371 + 1$ and $96256 = 96255 + 1$.  Furthermore, we encounter promising evidence in the form of the character of an element in conjugacy class 2A: the first term yields $q^{-1} - 492q + O(q^2)$, and there is a well-known modular function $\sqrt{j(2\tau)-1728} = q^{-1} - 492q - 22590q^3 - \cdots$ that is a Hauptmodul for $\Gamma_0(2|2)$.  However, a complication arises: the $q^2$ term in our character does not vanish, since the trace of 2A on $\underline{96255}$ is $4863$, not $-1$.  Despite the initial promising evidence, we find that the Baby Monster group does not yield moonshine-like behavior.  However, there is a central extension $2.\bB$, with an irreducible representation of dimension 96256, where the trace of 2A is zero.  By replacing the Baby Monster with its central extension, we eliminate this obstruction to assembling the characters into modular functions.  Similar behavior appears when considering the sporadic group $Fi_{24}$, where one needs to pass to a triple cover to get a representation of dimension 783, yielding a level 3 modular function $q^{-1} + 783q + \cdots$.  What these examples have in common is that both $2.\bB$ and $3.Fi_{24}$ are centralizers of elements of the Monster.

Norton employed these computations and many of his own to formulate the following conjecture:

\begin{conj} (Generalized Moonshine \cite{N87}, revised in \cite{N01})
There exists a rule that assigns to each element $g$ of the Monster simple group $\bM$ a graded projective representation $V(g) = \bigoplus_{n \in \bQ} V(g)_n$ of the centralizer $C_{\bM}(g)$, and to each pair $(g,h)$ of commuting elements of $\bM$ a holomorphic function $Z(g,h,\tau)$ on the complex upper half-plane $\fH$, satisfying the following conditions:
\begin{enumerate}
\item There is some lift $\tilde{h}$ of $h$ to a linear transformation on $V(g)$ such that
\[ Z(g,h,\tau) = \sum_{n \in \bQ} Tr(\tilde{h}|V(g)_n) q^{n-1}. \]
\item $Z(g,h,\tau)$ is invariant up to constant multiplication under simultaneous conjugation of the pair $(g,h)$ in $\bM$.
\item $Z(g,h,\tau)$ is either constant, or a Hauptmodul for some genus zero congruence group.
\item For any $\left( \begin{smallmatrix} a & b \\ c & d \end{smallmatrix} \right) \in SL_2(\bZ)$, $Z(g^a h^c, g^b h^d, \tau)$ is proportional to $Z(g,h, \frac{a \tau + b}{c \tau + d})$.
\item $Z(g,h,\tau) = J(\tau) = \frac{E_4(\tau)^3}{\Delta(\tau)} - 744 = q^{-1} + 196884 q + \cdots$ if and only if $g=h=1 \in \bM$.
\end{enumerate}
\end{conj}

This is a generalization of the Monstrous Moonshine Conjecture, in the sense that fixing $g=1$ yields an assertion that one has a graded representation of $\bM$ whose characters form Hauptmoduln.  It also subsumes the example observation above, in the sense that if $q^{-1} + 4372q + 96256q^2 + \cdots$ is the Moonshine character of an element $g$ in class 2A of the Monster, then $Z(g,1,\tau)$ is proportional to $q^{-1/2} + 4372q^{1/2} + 96256q + \cdots$, and these coefficients decompose into well-behaved representations of $2.\bB$, the centralizer of an element in class 2A of the Monster.

In contrast to the situation with the original Monstrous Moonshine conjecture, we do not have a list of conjugacy classes of commuting pairs of elements in $\bM$, or a list of character tables of central extensions of centralizers of elements at our disposal.  Conversations with group theorists suggest that these explicit data are still beyond current technology, so at the moment, we cannot solve the conjecture by computational brute force.  Instead, one needs to find an interpretation of the spaces $V(g)$ and the functions $Z$, and one must seek general methods for proving that traces like $Z$ yield Hauptmoduln, and for showing that traces are $SL_2(\bZ)$-compatible.

Shortly after Norton published this conjecture, Dixon, Ginsparg, and Harvey found a physical interpretation for claims 1,2,4, and 5: For each $g \in \bM$, the space $V(g)$ is the ``Hilbert space twisted by $g$'' in a conformal field theory with $\bM$ symmetry, and $Z(g,h)$ is a genus one partition function for a torus twisted in ``space'' by $g$ and in ``time'' by $h$ \cite{DGH88}.  In algebraic language, we would say that for each $g \in \bM$, we have a distinguished irreducible $g$-twisted module $V^\natural(g)$ of the Monster vertex operator algebra, and one may hope that such objects exist and are unique up to isomorphism.

The first major mathematical breakthrough toward a solution of this conjecture then came in \cite{DLM00}, with the following results:
\begin{enumerate}
\item Theorem 10.3 asserts that for any $g \in \bM$, there exists an irreducible $g$-twisted $V^\natural$-module, and it is unique up to isomorphism.  By Schur's Lemma, such a module $V^\natural(g)$ admits a canonical projective action of $C_{\bM}(g)$.  By general nonsense about conjugation, this resolves Claim 2, at least on the level of formal power series.
\item Theorem 8.1 implies the series $\sum_{n \in \bQ} Tr(\tilde{h}|V(g)_n) q^{n-1}$ converges absolutely uniformly on compact subsets of the complex upper half-plane to a holomorphic function, for any lift $\tilde{h}$ of an element $h \in C_{\bM}(g)$.  We can therefore define $Z(g,h,\tau)$ to be this function for some choice of $\tilde{h}$.  This resolves Claim 1, and translates the power-series version of Claim 2 to a holomorphic function version.
\item Theorem 10.3 also contains an assertion that implies $Z(g,1,\tau)$ is a constant multiple of $Z(1,g,-1/\tau)$.  This resolves Claim 5, essentially by the fact that $Z(1,g,\tau)$ has level greater than 1 for any $g \neq 1$.
\item Theorem 10.1 reduces Claim 4 to an conjecture about $g$-rationality of $V^\natural$, i.e., that any $g$-twisted $V^\natural$-module is isomorphic to a direct sum of copies of $V^\natural(g)$.
\end{enumerate}
Claim 4 was recently resolved: By Corollary 5.26 of \cite{CM16}, $g$-rationality holds for any simple regular non-negatively graded vertex operator algebra with an invariant bilinear form, and $V^\natural$ satisfies these properties.  The final remaining claim is Claim 3, the Hauptmodul-or-constant property, and this is the subject of this paper.

Under the current interpretation, Claims 1,2,4, and 5 are in fact general claims about non-negatively graded holomorphic $C_2$-cofinite vertex operator algebras and their automorphisms.  That is, these claims hold true for an infinite class of examples, and don't have much to do with the Monster, except for the fact that $V^\natural$ is a particularly well-behaved vertex operator algebra.  Claim 3 has a more exceptional, moonshine-like quality, since there are only finitely many Hauptmoduln, and the restriction on level gives an indication of smallness.  For this reason, a solution cannot be expected to appear directly from general facts about vertex operator algebras or conformal field theory.  Perhaps unsurprisingly, this is the only claim that was not given a physical interpretation by Dixon-Ginsparg-Harvey.  Proving such an exceptional claim seems to require a specialized strategy, and this is where we come to our second major breakthrough.

In an unpublished preprint \cite{H03}, H\"ohn resolved the Hauptmodul claim for the case where $g$ is an element in the conjugacy class 2A.  The resolution of this case is by itself a substantial advance, but more importantly, H\"ohn's preprint contained a general strategy for proving that Generalized Moonshine trace functions are Hauptmoduln.  Alternatively, one may say that H\"ohn worked out precisely the ingredients needed to extend Borcherds's proof of the Monstrous Moonshine Conjecture to the general case in question.  We shall call this the Borcherds-H\"ohn program:
\begin{enumerate}
\item Build an abelian intertwining algebra structure on the direct sum $\bigoplus_{i=0}^{|g|-1} V^\natural(g^i)$.
\item Add a spacetime torus: Make a conformal vertex algebra of central charge 26 by taking a graded tensor product with the abelian intertwining algebra of a suitable hyperbolic lattice.
\item Quantize: Apply a bosonic string quantization functor to the conformal vertex algebra.  This yields a Lie algebra $\fm_g$ equipped with a canonical projective action of $C_{\bM}(g)$ by automorphisms.
\item Generate a Lie algebra $L_g$ whose denominator identity is some automorphic infinite product.  This Lie algebra has a ``nice shape'', in the sense that its simple roots and homology are well-controlled.
\item Comparison: Show that the Lie algebra $\fm_g$ is isomorphic to $L_g$.  By \textit{transport de structure}, we obtain a Lie algebra with both a group action and nice shape.
\item Hauptmodul conclusion: Use the twisted denominator identity to produce recursion relations on the characters that are strong enough to conclude that the characters are Hauptmoduln.
\end{enumerate}

When $g=1$, this program gives an outline of Borcherds's proof \cite{B92}, with the first step given by Frenkel-Lepowsky-Meurman's construction of $V^\natural$.  When $g$ is an element of type 2A, this program gives an outline of H\"ohn's proof of Generalized Moonshine for $2.\bB$ \cite{H03}.  In both of those cases, the last step was settled by combining recursion relations with an explicit computation of the first few terms, to match the proposed functions with the actual characters.  In general, we do not have character tables for central extensions of centralizers in the Monster, so such explicit computations are out of the question.  Instead, we found a general Hauptmodul criterion involving equivariant Hecke operators in \cite{GM1}, generalizing the results of \cite{CG97}.  While Cummins-Gannon could be applied to remove the computations from Borcherds's proof of the Monstrous Moonshine conjecture \cite{B92}, our methods could be used to eliminate explicit computations for all cases we expected to encounter, including H\"ohn's proof.  This settles the last step, conditional on the constructions in the previous steps.

The step with $L_g$ was resolved in \cite{GM2}, i.e., we constructed automorphic infinite products and candidate Lie algebras for each element of the Monster.  We also defined a precise condition, called ``Fricke-compatibility'', on a projective action of $C_{\bM}(g)$ on $L_g$ that suffices for showing that the Hauptmodul condition holds for $g$ of Fricke type.  The question of producing such an action is essentially equivalent to the question of constructing $\fm_g$ and showing that $\fm_g \cong L_g$, and this is precisely what we resolve in this paper.

The first step is the most technically challenging.  For example, in the case $g=1$, this is the construction of the moonshine module $V^\natural$, which takes up the whole book \cite{FLM88}.  Before my recent project with Miyamoto culminating in \cite{CM16}, I had formed an elaborate plan to construct suitable abelian intertwining algebras using the theory of conformal blocks on twisted nodal curves together with some work eliminating homological obstructions.  Some of this theory was sketched in my dissertation \cite{C07}, and the third paper in this series \cite{GM3} is intended to present a full development.  I believe the methods are still useful in a general mathematical sense, and the first installment of results should appear in the near future.

We therefore come to the third major breakthrough, which rendered my plan superfluous for the purpose of proving Generalized Moonshine.  In \cite{vEMS}, the authors showed that there is an abelian intertwining algebra structure on a sum of twisted modules of a well-behaved vertex operator algebra, by an ingenious application of the Verlinde formula \cite{H08} together with the fixed-point regularity results of \cite{CM16}.  Furthermore, they determined the precise characters of modules of the fixed point vertex algebra $(V^\natural)^{\langle g \rangle}$.  Thus, their results essentially establish the first step, and go a long way toward the comparison step.

We now describe how the results of this paper fit in.  In order to construct $\fm_g$, we need a functorial way to pass from abelian intertwining algebras to Lie algebras.  We will employ an ``add a spacetime torus and quantize'' functor that has existed in some form for over 40 years, but we supply some necessary details that have not appeared in the literature.  Adding a torus is just a degree-wise tensor product with the abelian intertwining algebra of a well-chosen hyperbolic rational lattice.  The quantization functor takes conformal vertex algebras of central charge 26 to Lie algebras.  We exploit the ``cancellation of oscillators'' phenomenon to describe the homogeneous spaces of the Lie algebras explicitly in terms of the input object.  This theory was one of the main engines behind Frenkel's bound on root multiplicities of the Lie algebra whose simple roots form the Leech Lattice as its Dynkin diagram \cite{F85}, Borcherds's proof of the Monstrous Moonshine conjecture, and Hoehn's proof of Generalized Moonshine in the 2A case.

Once $\fm_g$ is constructed, we show that it is a Borcherds-Kac-Moody algebra by introducing a new criterion that is particularly well-suited to considering the output of the quantization functor.  Then, the homogeneous subspaces become root spaces, and they are canonically identified with certain subspaces of twisted modules as representations of some central extension of $C_{\bM}(g)$.  The comparison isomorphism $\fm_g \cong L_g$ then follows from a comparison of root multiplicities, and this is made possible by the precise character work in \cite{vEMS}. 

The reader may be interested to know just how much information about $V^\natural$ and the Monster was necessary for our proof to work.  In fact, we need almost no fine structure at all.  By examining the proof in this paper, together with the work on which it depends, we see that in fact we have the following result:

\begin{thm}
Let $V$ be a holomorphic $C_2$-cofinite vertex operator algebra of central charge 24 whose character is $J(\tau)$, and let $G$ be a finite order group of automorphisms of $V$.  Then any rule assigning to each $g \in G$ the irreducible $g$-twisted module $V(g)$, and to each commuting pair $(g,h)$ the function $Z(g,h;\tau) = \Tr(\tilde{h} q^{L_0-1}|V(g))$ for some lift $\tilde{h} \in \widetilde{C_G(g)}$ of $h$ to a linear transformation on $V(g)$, satisfies the conditions of Norton's conjecture, with $G$ in place of $\bM$.
\end{thm}

This is quite similar to the situation with Borcherds's proof of the Monstrous moonshine conjectures - one only needs central charge 24, an invariant bilinear form, and the character $J$ to get completely replicable functions  (which were later shown to be either highly degenerate or Hauptmoduln in \cite{CG97}).  However, without the construction of $V^\natural$ in \cite{FLM88}, this result loses most of its value, since the connection to the Monster is lost.  As it happens, Frenkel-Lepowsky-Meurman conjectured in section 3 of the introduction of \textit{loc. cit.} that any $V$ satisfying (more or less) the hypotheses in this theorem is isomorphic to $V^\natural$, so if we assume a suitable variant of the FLM Uniqueness Conjecture, then the statement given here becomes essentially the same as our Main Theorem (Theorem \ref{thm:main}).

We note that there is previous work on Generalized Moonshine that establishes some Hautpmodul results for twisted modules - perhaps the most advanced was \cite{IT02}.  However, these results were conditional on both the Uniqueness Conjecture and the construction of holomorphic orbifolds established in \cite{vEMS}.

We do not consider this the end of the Generalized Moonshine story.  There are conjectural refinements of our main theorem, and intriguing conjectural relations between Generalized Moonshine and other fields of mathematics and physics.  We focus on three further questions:
\begin{enumerate}
\item Norton has suggested stronger versions of this conjecture based on computational evidence \cite{N01}.  In particular, he asserts that the constant ambiguity can be reduced to a root of unity (in particular a 24th root), and that for any $(g,h)$, there is some rational $k$ such that each coefficient of $Z(g,h,\tau+k)$ generates a field extension of $\bQ$ with Galois group of exponent 2.  Terry Gannon has proposed a further refinement of the root-of-unity restriction, asserting that the proportionality constants are completely controlled by a 3-cocycle in the cohomology of $\bM$, with coefficients in $U(1)$.  The precise conjectural relationship between a suitable 3-cocycle and the proportionality constants is detailed in section 2 of \cite{GPV}.  Such a cocycle is expected from orbifold conformal field theory considerations \cite{DPR90}, and an analogue of this refinement for $M_{24}$ moonshine was established in \cite{GPRV}.  We note that this refinement is a general statement about twistings of holomorphic $C_2$-cofinite vertex operator algebras, and does not have a substantial connection to the Hauptmodul problem we solve in this paper.
\item Recall that almost all of the claims in the Generalized Moonshine conjecture were given physical interpretations in \cite{DGH88}, and the remaining claim concerns Hauptmoduln.  As it happens, there is a conjectural connection between Hauptmoduln and physics through Rademacher sums, as suggested in \cite{DF11}.  The general picture is that AdS/CFT correspondence seems to relate 2+1-dimensional gravity in the extreme quantum regime with conformal field theories satisfying smallness conditions.  In particular, Duncan and Frenkel propose the existence of a class of second-quantized twisted chiral quantum gravity theories that can be manipulated to produce the monstrous Lie algebras $\fm_g$ constructed in this paper.
\item There are intriguing conjectural connections between moonshine and equivariant elliptic cohomology.  The folklore idea is that elliptic cohomology is somehow controlled by, or is a shadow of, conformal field theory.  Concretely, an element of the elliptic cohomology of the classifying space of a finite group $G$ is a datum that includes a family of modular forms parametrized by a pair of commuting elements of $G$, satisfying an $SL_2(\bZ)$-compatibility condition.  It was noted quite early (and I am not sure who was first - perhaps Mason) that the weight zero version of this condition is precisely given as claim 4 in the Generalized Moonshine conjecture.   While elliptic cohomology does not play a direct role in our proof of the Generalized Moonshine conjecture, it was essential in an inspirational sense.  Indeed, equivariant Hecke operators, seen as power operations in equivariant elliptic cohomology, were the key to the formulation of a good Hauptmodul criterion in \cite{GM1}.  I believe the precise formula was first written in \cite{G09}, where one may also find a topological interpretation of replicability.  This piece of history, together with hints relating the formation of $\fm_g$ to exponential operations, suggests that the Hauptmodul claim may be amenable to a topological interpretation.
\end{enumerate}
We would be interested in seeing any developments toward the refined versions of this conjecture, and the establishment of stronger connections to other mathematical disciplines.

\subsection{Overview}

In section 2, we introduce some vertex algebraic constructions, such as abelian intertwining algebras and the ``add a torus'' functors.  The main purpose of this section is to take the results of \cite{vEMS}, and transmute their construction into a form that is amenable to quantization.

In section 3, we construct Lie algebras, using a generalization of Borcherds's method for constructing the Monster Lie algebra.  Due to the functorial nature of the construction, we obtain canonical group actions by homogeneous automorphisms.  We show that for Fricke elements, the resulting Lie algebras are necessarily Borcherds-Kac-Moody, and we show that the root multiplicities match the dimensions of certain eigenspaces in twisted modules.

In section 4, we compare root multiplicities.  For Fricke elements, this then yields isomorphisms with the Lie algebras constructed in \cite{GM2}, and we find that characters of centralizers are Hauptmoduln.  This then implies the full conjecture.

\subsection{Acknowledgments}

The author would like to thank Richard Borcherds, Jethro van Ekeren, Jeff Harvey, Yoshitake Hashimoto, Gerald H\"ohn, Satoshi Kondo, Gregory Moore, Lubo\v{s} Motl, Yu Nakayama, Simon Norton, and Nils Scheithauer for helpful conversations and advice.  Kazuya Kawasetsu, Sven M\"oller, and Hiromichi Yamada offered suggestions and corrections to an earlier version of this article.

This research was partly supported by NSF grant DMS-0354321,  JSPS Kakenhi Grant-in-Aid for Young Scientists (B) 24740005, the Program to Disseminate Tenure Tracking System, MEXT, Japan, and the World Premier International Research Center Initiative (WPI Initiative), MEXT, Japan.

\section{Vertex algebraic input}

Algebraic structures related to vertex algebras are essential to moonshine, but we will not need to spend much time on fine details, since most of the technical aspects have been worked out.  We will review some basic facts, then pursue some consequences of the work of \cite{vEMS}.

We will use the notation $e(s)$ to denote the normalized exponential $e^{2\pi i s}$.

\subsection{Vertex algebras and twisted modules}

\begin{defn}
A vertex algebra over $\cC$ is a complex vector space $V$ equipped with a distinguished vector $\unit \in V$, a distinguished linear transformation $T: V \to V$, and a left multiplication map $Y: V \to (\End V)[[z,z^{-1}]]$, written $Y(a,z) = \sum_{n \in \bZ} a_n z^{-n-1}$ satisfying the following conditions:
\begin{enumerate}
\item $Y(\unit, z) = id_V = id_V z^0$, and $Y(a,z)\unit \in a + zV[[z]]$.
\item For any $a,b \in V$, $Y(a,z)b \in V((z))$, i.e., $a_n b = 0$ for $n$ sufficiently large.
\item $[T, Y(a,z)] = \frac{d}{dz} Y(a,z)$
\item The Jacobi identity: for any $a,b \in V$, 
\[ x^{-1} \delta\left(\frac{y-z}{x}\right) Y(a,y) Y(b,z) - x^{-1} \delta\left(\frac{z-y}{-x}\right) Y(b,z) Y(a,y) =  
z^{-1} \delta\left(\frac{y-x}{z}\right) Y(Y(a,x)b,z) \]
where $\delta(z) = \sum_{n \in \bZ} z^n$ and $\delta(\frac{y-z}{x})$ is expanded as a formal power series in $z$.
\end{enumerate}
An automorphism of a vertex algebra $V$ is a linear transformation $g: V \to V$ that fixes $\unit$, and satisfies $Y(ga,z)gb = gY(a,z)b$ for all $a,b \in V$.
\end{defn}

\begin{defn}
A conformal vertex algebra of central charge $c \in \bC$ is a vertex algebra equipped with a distinguished nonzero vector $\omega$, satisfying the following conditions:
\begin{enumerate}
\item If we write $Y(\omega,z) = \sum_{n \in \bZ} L_n z^{-n-2}$, then the coefficients $L_n \in \End V$ satisfy the Virasoro relations:
\[ [L_m, L_n] = (m-n)L_{m+n} + c \frac{m^3-m}{12} \delta_{m+n,0} id_V \]
\item $L_0$ acts semisimply on $V$, with integer eigenvalues.
\item $L_{-1} = T$.
\end{enumerate}
An automorphism of a conformal vertex algebra is an automorphism of the underlying vertex algebra that fixes $\omega$.  A vertex operator algebra is a conformal vertex algebra for which the eigenvalues of $L_0$ are bounded below and have finite multiplicity.
\end{defn}

We will only consider one vertex operator algebra in this paper, namely the Monster vertex operator algebra $V^\natural$ that was constructed in \cite{FLM88}.

Vertex algebras have a more complicated representation theory than ordinary algebras, because modules may be twisted by automorphisms of the vertex algebra.

\begin{defn}
Let $V$ be a vertex algebra over $\bC$, and let $g$ be an automorphism of order $n$.  Then a $g$-twisted $V$-module is a complex vector space $M$ equipped with an action map $Y^M: V \to (\End M)[[z^{1/n}, z^{-1/n}]]$, written $Y^M(a,z) = \sum_{k \in \frac{1}{n} \bZ} a_k z^{-k-1}$ satisfying the following conditions:
\begin{enumerate}
\item $Y^M(\unit,z) = id_M$
\item For any $a \in V$, $u \in M$, $Y^M(a,z)u \in M((z^{1/n}))$, i.e., $a_k u = 0$ if $k$ is sufficiently large.
\item If $a \in V$ satisfies $ga = e(k/n)a$ for some $k \in \bZ$, then $Y^M(a,z) \in z^{k/n}(\End M)[[z,z^{-1}]]$.
\item The twisted Jacobi identity holds: if $ga = e(k/n)a$, then for any $b \in V$,
\[ \begin{aligned}
x^{-1} \delta\left(\frac{y-z}{x}\right) &Y^M(a,y) Y^M(b,z) - x^{-1} \delta\left(\frac{z-y}{-x}\right) Y^M(b,z) Y^M(a,y) =  \\
&= z^{-1} \left(\frac{y-x}{z} \right)^{-k/n} \delta\left(\frac{y-x}{z}\right) Y^M(Y(a,x)b,z)
\end{aligned} \]
\end{enumerate}
If $V$ is a vertex operator algebra, these conditions define the notion of ``weak $g$-twisted $V$-module''.  An ``admissible $g$-twisted $V$-module'' is a weak $g$-twisted $V$-module that admits a $\frac{1}{n}\bZ_{\geq 0}$-grading that is compatible with the $L_0$-grading on $V$.  An ``ordinary $g$-twisted $V$-module'' is an admissible $g$-twisted $V$-module for which $L_0$ acts semisimply, with finite dimensional eigenspaces, and eigenvalues that are bounded below in each coset of $\bZ$ in $\bC$.  When $g=1$, we replace ``$g$-twisted $V$-module'' with ``$V$-module''.
\end{defn}

\begin{defn}
Here are a few more technical conditions about vertex algebras.  We won't use them in any specific way, but they appear as conditions in the theorems we need.
\begin{enumerate}
\item A vertex algebra $V$ is $C_2$-cofinite if the subspace $C_2(V)$ spanned by $\{ a_2 b | a,b \in V\}$ satisfies the property that $V/C_2(V)$ is finite dimensional.
\item A vertex operator algebra $V$ is $g$-rational if all admissible $g$-twisted $V$-modules are direct sums of irreducible ordinary $g$-twisted $V$-modules.
\item A vertex operator algebra $V$ is holomorphic if all admissible $V$-modules are isomorphic to direct sums of $V$.
\item A vertex operator algebra $V$ is of CFT type if $L_0$ has only non-negative eigenvalues, and the kernel of $L_0$ is spanned by $\unit$.
\end{enumerate}
\end{defn}

\begin{rem} \label{rem:automorphisms-of-twisted-modules}
Theorem 10.3 in \cite{DLM00} asserts that if a vertex operator algebra $V$ is $C_2$-cofinite and holomorphic, then there exists a unique irreducible $g$-twisted $V$-module, which we will write as $V(g)$.  For any automorphism $h$ of $V$, one may change the action of $V$ on the underlying vector space of $V(g)$ to produce a module isomorphic to $V(h^{-1}gh)$.  By Schur's Lemma, we obtain a canonical projective action of $C_{\Aut V}(g)$ on the underlying vector space of $V(g)$, which is compatible with the action on $V$ (but it is not an action by twisted module automorphisms, since those are just scalars).  
\end{rem}

\begin{defn}
We write $\Aut V(g)$ to denote the group of vector space automorphisms of $V(g)$ that are compatible with automorphisms of $V$ - it is a central extension of $C_{\Aut V}(g)$ by $\bC^\times$.
\end{defn}

\subsection{Abelian intertwining algebras}

Dong and Lepowsky introduced abelian intertwining algebras essentially as vertex algebras with some monodromy obstructing commutativity.  We describe them briefly, and we describe the theorem we need from \cite{vEMS}.

\begin{defn} (\cite{M52})
Let $A$ be an abelian group.  The Eilenberg-Mac Lane cochain complex with coefficients in $\bC^\times$ has the following description in low degree:
\[ \xymatrix{ K^1 \ar[r]^-{d^1} & K^2 = \{ f: A \times A \to \bC^\times \} \ar[r]^-{d^2} &  K^3 = \{ (F, \Omega): A^{\oplus 3} \times A^{\oplus 2} \to \bC^\times \} \ar[r]^-{d^3} & \cdots } \]
where:
\begin{enumerate}
\item $K^1 = \{ \phi: A \to \bC^\times \}$ and $K^2 = \{ f: A \times A \to \bC^\times \}$ are groups of set-theoretic maps.
\item $K^3 = \{ (F, \Omega): A^{\oplus 3} \times A^{\oplus 2} \to \bC^\times \}$ is a group of pairs of set-theoretic maps.
\item The map $d^1: K^1 \to K^2$ is given by the usual group cohomology coboundary $\phi \mapsto d^1 \phi$, defined by $d^1\phi(i,j) = \phi(j) - \phi(i+j) + \phi(i)$.
\item The map $d^2: K^2 \to K^3$ is given by the group cohomology coboundary $d^2 f(i,j,k) = f(j,k) - f(i+j,k) + f(i,j+k) - f(i,j)$ together with the antisymmetrizer: $(i,j) \mapsto \frac{f(i,j)}{f(j,i)}$.
\item The map $d^3: K^3 \to K^4$ vanishes if and only if the following conditions are satisfied: 
\begin{enumerate}
\item $F(i,j,k) F(i,j+k,\ell) F(j,k,\ell) = F(i+j,k,\ell) F(i,j,k+\ell)$ for all $i,j,k,\ell \in A$
\item $F(i,j,k)^{-1}\Omega(i,j+k) F(j,k,i)^{-1} = \Omega(i,j) F(j,i,k)^{-1} \Omega(i,k)$
\item $F(i,j,k) \Omega(i+j,k) F(k,i,j) = \Omega(j,k) F(i,k,j) \Omega(i,k)$
\end{enumerate}
\end{enumerate}
The kernel of $d^n$ is written $Z^n_{ab}(A,\bC^\times)$, and elements are called abelian $n$-cocycles.   The cohomology of the Eilenberg-MacLane complex is called abelian cohomology, and written $H^n_{ab}(A,\bC^\times)$.  An abelian 3-cocycle is normalized if $F(i,j,0) = F(i,0,k) = F(0,j,k) = 1$ and $\Omega(i,0) = \Omega(0,j) = 1$ for all $i,j,k \in A$.
\end{defn}

We note that any cohomology class has a normalized representative.

\begin{lem} \label{lem:cocycle-trace} (\cite{M52} Theorem 3)
If $(F, \Omega)$ is an abelian 3-cocycle, then the map $Q: A \to \bC^\times$ defined by $i \mapsto \Omega(i,i)$ is a quadratic function, i.e., $Q(i) = Q(-i)$ for all $i \in A$, and $\frac{Q(i+j+k) Q(i) Q(j) Q(k)}{Q(i+j) Q(i+k) Q(j+k)} = 1$ for all $i,j,k \in A$.  Furthermore, the trace map $(F, \Omega) \mapsto (i \mapsto \Omega(i,i))$ induces a bijection from $H^3_{ab}(A, \bC^\times)$ to the set of $\bC^\times$-valued quadratic functions on $A$.
\end{lem}
\begin{proof}
This is proved as \cite{EM54}, theorem 26.1.
\end{proof}

\begin{defn}
Let $A$ be an abelian group, let $(F,\Omega)$ be an abelian 3-cocycle on $A$.  We write $q_\Omega: A \to \bC/\bZ$ for the unique quadratic map such that $e(q_\Omega(a)) = \Omega(a,a)$ for all $a \in A$, and we write $b_\Omega: A \times A \to \bC/\bZ$ for the induced bilinear form.
\end{defn}

We take the following definition from \cite{DL93}

\begin{defn}
Let $A$ be an abelian group, let $(F,\Omega)$ be a normalized abelian 3-cocycle on $A$.  Then an abelian intertwining algebra of level $N \in \bZ_{\geq 1}$ and central charge $c \in \bQ$ associated to the datum $(A, F, \Omega)$ is a complex vector space $V$ equipped with
\begin{enumerate}
\item a $\frac{1}{N}\bZ \times A$-grading $V = \bigoplus_{n \in \frac{1}{N} \bZ} V_n = \bigoplus_{i \in A} V^i = \bigoplus_{n \in \frac{1}{N} \bZ, i \in A} V_n^i$
\item a left-multiplication $Y: V \to (\End V)[[z^{1/N},z^{-1/N}]]$ written $Y(a,z) = \sum_{n \in \frac{1}{N}\bZ} a_n z^{-n-1}$, and
\item distinguished vectors $\unit \in V_0^0$ and $\omega \in V_2^0$
\end{enumerate}
satisfying the following conditions for any $i,j,k \in A$, $a \in V^i$, $b \in V^j$, $u \in V^k$, and $n \in \frac{1}{N}\bZ$:
\begin{enumerate}
\item $a_n b \in V^{i+j}$.
\item $a_n b = 0$ for $n$ sufficiently large.
\item $Y(\unit, z)a = a$
\item $Y(a,z)\unit \in a + zV[[z]]$.
\item  $Y(a,z)b = \sum_{k \in b_\Omega(i,j) + \bZ} a_k b z^{-k-1}$.
\item The Jacobi identity:
\[ \begin{aligned}
 x^{-1} &\left(\frac{y-z}{x} \right)^{b_\Omega(i,j)} \delta\left(\frac{y-z}{x}\right) Y(a,y) Y(b,z) u \\
 &- B(i,j,k) x^{-1} \left(\frac{z-y}{e^{i\pi}x} \right)^{b_\Omega(i,j)} \delta\left(\frac{z-y}{-x}\right) Y(b,z) Y(a,y) u = \\
& \qquad = F(i,j,k) z^{-1} \delta\left(\frac{y-x}{z}\right) Y(Y(a,x)b,z) \left(\frac{y-x}{z} \right)^{-b_\Omega(i,k)} u 
\end{aligned}\]
where $B(i,j,k) = \frac{\Omega(i,j)F(i,j,k)}{F(j,i,k)}$. 
\item The coefficients of $Y(\omega,z) = \sum_{k \in \bZ} L_k z^{-k-2}$ satisfy the Virasoro relations at central charge $c$.
\item $L_0 a = na$ if $a \in V_n$.
\item $\frac{d}{dz} Y(a,z) = Y(L_{-1}a,z)$.
\end{enumerate}
\end{defn}

Given a suitably nondegenerate abelian intertwining algebra (namely one where left multiplication by elements of $V^i$ is not the zero map), one has an additional $\bC/\bZ$-valued quadratic form $q_\Delta$ on $A$, given by setting $q_\Delta(i)$ to be the conformal weight of $V^i$.  It is not hard to show that the associated bilinear forms satisfy $b_\Omega = - b_\Delta$ (see, e.g., the discussion in \cite{vEMS} after Theorem 4.2).  Furthermore, it is often the case that we have $q_\Omega = -q_\Delta$ - when this holds, we say that the abelian intertwining algebra is ``even''.

Dong and Lepowsky give the following generalization of the lattice vertex algebra construction: for any finite rank free abelian group $L$ with a $\frac{1}{N}\bZ$-valued quadratic form, we may produce the abelian intertwining algebra $V = V_L$ attached to $L$.  It is $L$-graded, and the $L$-homogeneous pieces are the usual modules $\pi_\lambda^{L \otimes \bC}$ of the Heisenberg vertex algebra $\pi_0^{L \otimes \bC}$ for the quadratic vector space $L \otimes \bC$ attached to vectors $\lambda \in L$.  This abelian intertwining algebra is even, i.e., it satisfies $q_\Delta = -q_\Omega$.  One needs to choose a cocycle $(F,\Omega)$ in the cohomology class of $q_\Omega$ to obtain an abelian intertwining algebra, but we have the following way to pass between different choices.

\begin{lem} \label{lem:twist-by-coboundary}
Given an abelian intertwining algebra $V$ with cocycle $(F,\Omega)$, and any function $f \in K^2$ satisfying $f(i,0) = f(0,j) = 1$, the tuple $(V,1,\omega, f \cdot Y)$ is an abelian intertwining algebra with cocycle $d_2f \cdot (F, \Omega)$.  Here, $f \cdot Y$ is the map $V \to (\End V)[[z^{1/N},z^{-1/N}]]$ that takes a pair $a \in V^i$, $b \in V^j$ to $f(i,j) \cdot Y(a,z)b$.
\end{lem}
\begin{proof}
This is Remark 12.23 in \cite{DL93}.
\end{proof}

\begin{defn}
Let $q$ be a quadratic form on an abelian group $A$ that takes values in $\frac{1}{N} \bZ/\bZ$.  We define $AIA^c_{(A,q)}$ to be the following groupoid: objects are abelian intertwining algebras of level $N$ with central charge $c$, associated to the datum $(A,F,\Omega)$ for some abelian 3-cocycle $(F,\Omega)$ such that $q_\Omega = e(q)$.  Morphisms are homogeneous isomorphisms of abelian intertwining algebras.
\end{defn}

By Lemma \ref{lem:twist-by-coboundary}, this groupoid is equipped with an action of the abelian 3-coboundary group by autofunctors.

The last structure we introduce is a graded tensor product.

\begin{defn}
Let $V$ be an object in $AIA^c_{(A,q)}$ with cocycle $(F,\Omega)$, and let $W$ be an object in $AIA^{c'}_{(A,q')}$ with cocycle $(F',\Omega')$.  Then we define $V \otimes^A W$ to be the following object in $AIA^{c+c'}_{(A,qq')}$ with cocycle $(FF',\Omega\Omega')$:
\begin{enumerate}
\item The underlying graded vector space is $\bigoplus_{i \in A} V^i \otimes W^i$
\item Multiplication is induced by tensor product maps.  That is, $Y(a \otimes a',z)(b \otimes b') = Y(a,z)b \otimes Y(a',z)b'$ for $a,b \in V$, and $a',b' \in W$.
\item $\unit_{V \otimes^A W} = \unit_V \otimes \unit_W$
\item $\omega_{V \otimes^A W} = \omega_V \otimes \unit_W + \unit_V \otimes \omega_W$
\end{enumerate}
\end{defn}

The following theorem describes the results of van Ekeren, M\"oller, and Scheithauer that we need.

\begin{thm} \label{thm:vEMS}
Let $V$ be a holomorphic $C_2$-cofinite vertex operator algebra of CFT type, let $g$ be an automorphism of $V$ of order $n \in \bZ_{\geq 1}$, and let $G = \langle g \rangle$ be the cyclic group generated by $g$.  Suppose the $L_0$-spectrum of the irreducible twisted module $V(g^i)$ is strictly positive whenever $g^i \neq 1$.  Then there is a distinguished element $t \in \bZ/n\bZ$ and a unique assignment of homomorphisms $\phi_i: G \to \Aut V(g^i)$ satisfying the following conditions:
\begin{enumerate}
\item If we let $V^{(i,j)} = \{ v \in V(g^i) | \phi_i(g)v = e(j/n)v \}$, then the set $\{ V^{(i,j)} \}_{i,j \in \bZ/n\bZ}$ forms a complete list of isomorphism classes of irreducible $V^G$-modules.  They are simple currents, whose fusion rules satisfy $V^{(i,j)} \boxtimes V^{(k,l)} \cong V^{(i+k,j+l+c_{2t}(i,k))}$, where $c_{2t}$ is the ``add with carry'' 2-cocycle $\bZ/n\bZ \times \bZ/n\bZ \to \bZ/n\bZ$ given by $c_{2t}(i,k) = \begin{cases} 0 & i_n + k_n < n \\ 2t & i_n + k_n \geq n \end{cases}$.  Here, $i_n$ and $k_n$ denote the unique elements of $\{0,\ldots,n-1\}$ to which $i$ and $k$ reduce modulo $n$, respectively, and we let $D$ denote the central extension of $\bZ/n\bZ$ by $\bZ/n\bZ$ defined by the 2-cocycle $c_{2t}$.
\item There exists an abelian intertwining algebra ${}^g V \cong \bigoplus_{i \in \bZ/n\bZ} V(g^i) \cong \bigoplus_{i,j \in \bZ/n\bZ} V^{(i,j)}$, whose multiplication is compatible with the fusion rules on $V^{(i,j)}$, graded by the abelian group $D$, with $q_\Omega(i,j) = -\frac{ij}{n} - \frac{i_n^2 t_n}{n^2}$.  Furthermore,  ${}^g V$ is even, i.e., $q_\Delta(i,j) = -q_\Omega(i,j)$.
\end{enumerate}
Furthermore, if we write
\[ \begin{aligned}
T(v,i,j,\tau) &= \Tr(o(v) \phi_i(g^j) q^{L_0-c/24}|V(g^i))\\
T_{(i,j)}(v;\tau) &= \Tr(o(v) q^{L_0-c/24}|V^{(i,j)})
\end{aligned}\]
for all $v \in V^G$, $i,j \in \bZ/n\bZ$ and $q = e(\tau)$, where $o(v)$ is the zero-mode map (which is $v_{wt(v)-1}$ for homogeneous $v$), we obtain holomorphic functions on $\fH$ that satisfy the following modular transformation rules:
\[ \begin{aligned}
T(v,i,j,\tau+1) &= e(\frac{i_n^2 t_n}{n^2} - \frac{c}{24}) T(v,i,j+i,\tau) \\
T(v,i,j,\frac{-1}{\tau}) &= \tau^{wt[v]} e(\frac{-2t i_n j_n}{n^2}) T(v,j,-i,\tau) \\
T_{(i,j)}(v;\tau+1) &= e(\frac{ij}{n} + \frac{i_n^2 t_n}{n^2} - \frac{c}{24}) T_{(i,j)}(v;\tau) \\
T_{(i,j)}(v;\frac{-1}{\tau}) &= \frac{1}{n} \tau^{wt[v]} \sum_{k,l \in \bZ/n\bZ} e(-\frac{kj+il}{n}) e(\frac{-2ti_nk_n}{n^2}) T_{(k,l)}(v;\tau)
\end{aligned} \]
where $wt[v]$ denotes the conformal weight with respect to Zhu's bracket grading \cite{Z96}.
\end{thm}
\begin{proof}
We point out where the specific claims are given in \cite{vEMS}.  The identification of eigenspaces with irreducible $V^G$-modules is Proposition 5.2, and the fact that they are simple currents is Proposition 5.6.

The fusion rules, conformal weights and $S$-matrix entries are given in Theorem 5.12.  The modular transformation rules for $T(v,i,j,\tau)$ and $T_{(i,j)}(v;\tau)$ then follow from the identification of $S$-matrix entries and the conformal weights of modules.

The existence and properties of the abelian intertwining algebra are given in Theorem 4.1.
\end{proof}

We now show that this theorem can be applied to the object of our interest.  Recall that an element $g$ in the Monster is called Fricke if its McKay-Thompson series $T_g(\tau)$ is proportional to $T_g(-1/N\tau)$ for some $N$.  Note that this is a definition that depends on the construction of the moonshine module $V^\natural$ in \cite{FLM88}.  By Borcherds's proof of the Monstrous Moonshine Conjecture \cite{B92}, we have the following necessary and sufficient conditions for $g$ to be Fricke: $T_g$ has a pole at zero, or $T_g(\tau)$ is equal to $T_g(-1/N\tau)$ for some $N$.  The Fricke property is conjugation-invariant, since the McKay-Thompson series is a trace.  The conjectural enumeration of functions in \cite{CN79} (proved to be valid by Borcherds) shows that 120 of the 171 distinct McKay-Thompson series are Fricke, and 143 of the 194 conjugacy classes in the Monster are Fricke.

\begin{thm} \label{thm:Monster-voa-is-well-behaved}
The Monster vertex operator algebra $V^\natural$ is holomorphic, $C_2$-cofinite, and of CFT type.  For any nontrivial element $g$ of the Monster, the $L_0$-spectrum of $V(g)$ is strictly positive.  That is, the conditions of Theorem \ref{thm:vEMS} are satisfied for any automorphism of $V^\natural$.
\end{thm}
\begin{proof}
The holomorphic property is proved in \cite{D94}, the $C_2$-cofinite property is proved in \cite{DLM00}, and CFT type is clear from the construction \cite{FLM88}.  By Theorem 13.1 of \cite{DLM00}, the character of $V^\natural(g)$ is a constant multiple of $T_g(-1/\tau)$, where $T_g(\tau) = Tr(gq^{L_0-1}|V^\natural)$ is the McKay-Thompson series.  If $T_g$ is regular at $\tau = 0$ (i.e., $g$ is non-Fricke), then the lowest-weight space of $V(g)$ has $L_0$-eigenvalue at least 1.  If $T_g$ has a pole at $\tau = 0$ (i.e., $g$ is Fricke), then the pole has order $1/N$, where $N$ is the level of $T_g$, and the lowest-weight space of $V^\natural(g)$ has $L_0$-eigenvalue $1-1/N$.  We conclude that as long as $g$ is nontrivial, the lowest weight space has positive $L_0$-eigenvalue.
\end{proof}

We may then determine the precise graded dimension of twisted modules of $V^\natural$ - this eliminates the scalar ambiguity in \cite{DLM00}.

\begin{cor}
For any $g \in \bM$, hypothesis $A_g$ from section 4.4 of \cite{GM2} is satisfied, i.e., $\sum_{n \in \bQ} \dim V^\natural(g)_n q^{n-1} = T_g(-1/\tau)$.
\end{cor}
\begin{proof}
The left side of the equation is $T(\unit,1,0,\tau)$, and the right side is $T(\unit,0,1,-1/\tau)$, as defined in Theorem \ref{thm:vEMS}.  By Theorem \ref{thm:Monster-voa-is-well-behaved}, we may apply Theorem \ref{thm:vEMS} to $V^\natural$ and its twisted modules, so the two quantities are equal.  
\end{proof}

\begin{cor}
If $V = V^\natural$, and $g \in \bM$ has eigengroup $\Gamma_0(n|h)+e_1,e_2,\ldots$ (in the notation of \cite{CN79}), then the distinguished element $t \in \bZ/n\bZ$ given in Theorem \ref{thm:vEMS} is equal to $-n/h$ if $g$ is Fricke and $n/h$ if $g$ is non-Fricke.
\end{cor}
\begin{proof}
The value of $t$ is uniquely determined by the transformation $\tau \mapsto \tau+n$ on the graded dimension of $V^\natural(g)$.  Theorem \ref{thm:vEMS} gives us $T(\unit,1,0,\tau-n) = e(-\frac{t_n}{n}) T(\unit,1,0,\tau)$, and Section 5 of \cite{CN79} gives us the transformation rule $T_g(\tau)|_0 \left(\begin{smallmatrix} 1 & 0 \\ n & 1 \end{smallmatrix} \right) = e(\pm 1/h) T_g(\tau)$, where we have a plus sign if and only if $g$ is Fricke.  The latter M\"obius transformation conjugates to the $\tau \mapsto \tau-n$ transformation on $T(\unit,1,0,\tau)$, so we have $t_n = \mp n/h$.
\end{proof}

\subsection{Unrolling}

In Section 3.1 of \cite{GM2} we introduce vector-valued modular functions $\hat{F}$ and $F$, which are uniquely determined by the $SL_2(\bZ)$ action on McKay-Thompson series.  In section \ref{sect:comparisons}, we will need to match the functions $T(\unit,i,j,\tau)$ with $\hat{F}_{i,j}(\tau)$ and $T_{(i,j)}(\tau)$ with $F_{i,j}(\tau)$ in order to show that certain Lie algebras are isomorphic.  However, this is only possible when $t=0$ given our current setup, because that is the only way the modular transformation rules can match.  We will fix this problem by ``unrolling'' the abelian intertwining algebra into a $\bZ/N\bZ \times \bZ/N\bZ$-graded object in a unique way, for $N$ any multiple of $\frac{n^2}{(n,t)}$.  By passing to the unrolled object, the modular transformation rules match in all cases.

\begin{lem} \label{lem:pullback-algebra}
Let $V_A = \bigoplus_{a \in A} V_a$ be an abelian intertwining algebra with cocycle $(F_A,\Omega_A)$.  Let $(B,q)$ be an abelian group with $\bC/\bZ$-valued quadratic form, and let $\pi: (B,q) \to (A,(a \mapsto \Omega_A(a,a)))$ be the quotient by a null subgroup $I \subset B^\perp$.  Then there is a unique ``pullback'' abelian intertwining algebra $W_B = \bigoplus_{b \in B} W_b$ satisfying the following properties:
\begin{enumerate}
\item $W_b =  \begin{cases} V_a & b \in B, \pi(b)=a \\ 0 & b \notin B \end{cases}$
\item The multiplication map $W_b \otimes W_{b'} \to W_{b+b'}\{z\}$ is given by the multiplication $V_a \otimes V_{a'} \to V_{a+a'}\{z\}$ from $V_A$ when $b, b' \in B$, $\pi(b)= a$, and $\pi(b') = a'$.
\end{enumerate}
The cocycle $(F_B, \Omega_B)$ attached to $W_B$ is precisely the pullback of the cocycle $(F_A,\Omega_A)$ attached to $V_A$.
\end{lem}
\begin{proof}
The fact that $W_B$ satisfies the axioms of an abelian intertwining algebra follows straightforwardly from our definition of the multiplication operation - the Jacobi identity for $W_B$ reduces immediately to the Jacobi identity for $V_A$.
\end{proof}

\begin{lem} \label{lem:extension-by-zero}
If $i: (A,q_A) \to (B,q_B)$ is a quadratic embedding (i.e., $i$ is an injective homomorphism of abelian groups, and $q_B \circ i = q_A$), and $V_A$ is an $(A,q_A)$-graded abelian intertwining algebra, then there is a unique ``extension by zero'' abelian intertwining algebra $W_B = \bigoplus_{b \in B} W_b$ satisfying the following properties
\begin{enumerate}
\item $W_b =  \begin{cases} V_a & b = i(a) \\ 0 & b \notin i(A) \end{cases}$
\item The multiplication map $W_b \otimes W_{b'} \to W_{b+b'}\{z\}$ is given by the multiplication $V_a \otimes V_{a'} \to V_{a+a'}\{z\}$ from $V_A$ when $b = i(a)$, and $b' = i(a')$.  For all other values of $b$ and $b'$, at least one of the input spaces is zero, so multiplication is trivial.
\end{enumerate}
The cocycle $(F_B, \Omega_B)$ attached to $W_B$ is uniquely determined by $(F_A,\Omega_A)$ up to a coboundary that vanishes identically on $i(A)$.
\end{lem}
\begin{proof}
Omitted.
\end{proof}

\begin{defn}
Given a diagram $(A,q) \overset{\pi}{\twoheadleftarrow} (A',q') \overset{i}{\hookrightarrow} (A'',q'')$ of quadratic spaces, we define the unrolling functor $AIA^c_{(A,q)} \to AIA^c_{(A'',q'')}$ as the composite of the pullback given in Lemma \ref{lem:pullback-algebra} and the extension by zero given in Lemma \ref{lem:extension-by-zero}.
\end{defn}

The following proposition shows that we can make the modular transformation rules in Theorem \ref{thm:vEMS} somewhat cleaner by unrolling.

\begin{prop} \label{prop:unrolled-algebra}
Let $V$ be a holomorphic $C_2$-cofinite vertex operator algebra of CFT type, let $g$ be an automorphism of $V$ of order $n \in \bZ_{\geq 1}$, such that the $L_0$-spectrum of the irreducible twisted module $V(g^i)$ is strictly positive whenever $g^i \neq 1$.  Let $t \in \bZ/n\bZ$
be the distinguished element given by Theorem \ref{thm:vEMS}, and let $N = \frac{r n^2}{(n,t)}$ for some $r \in \bZ_{\geq 1}$.  Finally, let $B$ denote the group $\bZ/N\bZ \times \bZ/N\bZ$ with quadratic form $(a,b) \mapsto e(\frac{ab}{N})$.  Then the subgroup $B' \subset B$ generated by $(0, \frac{rn}{(n,t)})$ and $(1,\frac{t_n r}{(n,t)})$ admits a quadratic-form-preserving quotient map $B' \to D$ by $a(1,\frac{t_n r}{(n,t)}) + b(0, \frac{rn}{(n,t)}) \mapsto (a_n,b_n+2t_n\lfloor a/n \rfloor)$, and there exists a unique set of homomorphisms $\{ \psi_a: \bZ/N\bZ \to \Aut V(g^a) \}_{a = 0}^{N-1}$ such that the abelian intertwining algebra ${}^g_N W \cong \bigoplus_{(a,b) \in B} W^{(a,b)} \cong \bigoplus_{a=0}^{N-1} V(g^a)$ unrolled from ${}^gV$ satisfies the following conditions:
\begin{enumerate}
\item For each $a,b \in \bZ/N\bZ$, we have $W^{(a,b)} = \{ v \in V(g^a) | \psi_a(1)v = e(\frac{b}{N})v \}$, and when $W^{(a,b)} \neq 0$, it is a simple current (in particular, an irreducible module) for the fixed-point algebra $V^g$.
\item For any $a,b \in \bZ/N\bZ$, we have the equality $\psi_a(b) = \phi_a(g^b) e(b\frac{a_n t_n}{n^2} - b\frac{t_n \lfloor a/n \rfloor}{n})$ of automorphisms of $V(g^a)$.
\item If we define $T(a,b,\tau) = \Tr( \psi_a(b) q^{L_0-c/24}|V(g^a))$ and $T_{(a,b)}(\tau) = \Tr( q^{L_0-c/24}|W^{(a,b)})$ for all $a,b \in \bZ/N\bZ$ and $q = e(\tau)$, then we obtain holomorphic functions on $\fH$ that satisfy the following modular transformation rules:
\[ \begin{aligned}
T(a,b,\tau+1) &= e(-\frac{c}{24}) T(a,b+a,\tau) \\
T(a,b,\frac{-1}{\tau}) &= T(b,-a,\tau) \\
T_{(a,b)}(\tau+1) &= e(\frac{ab}{N} - \frac{c}{24}) T_{(a,b)}(\tau) \\
T_{(a,b)}(\frac{-1}{\tau}) &= \frac{1}{N} \sum_{p,q \in \bZ/n\bZ} e(-\frac{bp+aq}{N}) T_{(p,q)}(\tau)
\end{aligned} \]
In particular (following \cite{GM2} Lemma 3.6), if $c$ is a multiple of 24, then the collection $\{ T_{(a,b)} \}_{a,b \in \bZ/N\bZ}$ forms a vector-valued modular function of type $\rho_M$ for $M = I\!I_{1,1}(N)$, and the collection $\{T(a,b,\tau) \}_{a,b \in \bZ/N\bZ}$ describes a function on the moduli space $\cM_{Ell}^{\bZ/N\bZ}$ of elliptic curves with $\bZ/N\bZ$-torsors.
\end{enumerate}
\end{prop}
\begin{proof}
We first check that $B' \to D$ preserves the quadratic form: $q_{B'}(a(1,\frac{t_n r}{(n,t)}) + b(0, \frac{rn}{(n,t)})) = \frac{ab}{n} + \frac{a^2 t_n}{n^2} + \bZ$, while $q_D(a_n,b_n+2t_n\lfloor a/n \rfloor) = \frac{a_nb_n + 2a_n t_n\lfloor a/n \rfloor}{n} + \frac{a_n^2 t_n}{n^2} + \bZ$.  Because $a = n\lfloor a/n \rfloor + a_n$, we have $\frac{a^2 t_n}{n^2} + \bZ = \frac{a_n^2 t_n}{n^2} + \frac{2a_n t_n \lfloor a/n \rfloor}{n} + \bZ$, so the two expressions agree.

We therefore have a well-defined unrolling functor, and an abelian intertwining algebra ${}^g_N W$ whose $B$-homogeneous components and multiplication maps are either pulled back from ${}^g V$ or zero.  The homomorphisms $\psi_a$ are uniquely determined by the $B$-grading on ${}^g_N W$.  The simple current property follows from the first listed claim in \ref{thm:vEMS}.

We now check that $\psi_a(b) = \phi_a(g^b) e(b\frac{a_n t_n}{n^2} - b\frac{t_n \lfloor a/n \rfloor}{n})$ on $V(g^a)$.  For any $a(1,\frac{t_n r}{(n,t)}) + j(0, \frac{rn}{(n,t)}) \in B'$, we have $\psi_a(1)|_{W^{a(1,\frac{t_n r}{(n,t)}) + j(0, \frac{rn}{(n,t)})}} = e(\frac{at_n}{n^2} + \frac{j}{n})$.  The module is the pullback of $V^{(a_n, j_n + 2t_n\lfloor a/n\rfloor)}$, where $\phi_a(g)$ acts as $e(\frac{j_n + 2t_n\lfloor a/n \rfloor}{n})$.  Dividing the two answers, and taking the $b$th power yields the answer we want.

It remains to consider the modular transformation properties.  We reduce the first two equations to the last two equations by noting that $\{T(a,b,\tau)\}$ and $T_{(a,b)}(\tau)$ are related by a discrete Fourier transform.  That is, the equations
\[ \begin{aligned}
T(a,b,\tau) &= \sum_{k \in \bZ/N\bZ} e\left(\frac{bk}{N}\right) T_{(a,k)}(\tau) \\
T_{(a,b)}(\tau) &= \frac{1}{N} \sum_{k \in \bZ/N\bZ} e\left(\frac{-bk}{N} \right) T(a,k,\tau)
\end{aligned} \]
allow us to pass from one equation to the other - this is essentially Lemma 3.3 of \cite{GM2} but we multiply by the additional character $\left(\begin{smallmatrix} 1 & 1 \\ 0 & 1 \end{smallmatrix} \right) \mapsto e(-c/24)$, $\left(\begin{smallmatrix} 0 & -1 \\ 1 & 0 \end{smallmatrix} \right) \mapsto 1$ of $SL_2(\bZ)$.

Since $W^{(a(1,\frac{t_n r}{(n,t)}) + b(0, \frac{rn}{(n,t)}))} = V^{(a_n,b_n+2t_n\lfloor a/n \rfloor)}$, we have
\[ T_{(a(1,\frac{t_n r}{(n,t)}) + b(0, \frac{rn}{(n,t)}))}(\tau) = T_{(a_n,b_n+2t_n\lfloor a/n \rfloor)}(\unit,\tau) \]
for all $a,b \in \bZ/N\bZ$, and $T_{(a,b)}(\tau) = 0$ for $(a,b) \notin B'$.  By our computation of the quadratic form in the first paragraph of this proof, we find that
\[ \begin{aligned}
T_{(a(1,\frac{t_n r}{(n,t)}) + b(0, \frac{rn}{(n,t)}))}(\tau + 1) &= T_{(a_n,b_n+2t_n\lfloor a/n \rfloor)}(\unit,\tau+1) \\
&= e\left(\frac{a_n(b_n+2t_n\lfloor a/n \rfloor)}{n} + \frac{a_n^2 t_n}{n^2} - \frac{c}{24}\right) T_{(a_n,b_n+2t_n\lfloor a/n \rfloor)}(\unit,\tau) \end{aligned}\]
Then $\frac{a_n(b_n+2t_n\lfloor a/n \rfloor)}{n} + \frac{a_n^2 t_n}{n^2} = \frac{ab}{n} + \frac{a^2 t_n}{n^2} = \frac{abrn + a^2t_n r}{N(n,t)} = \frac{a\left(\frac{at_n r}{(n,t)} + \frac{brn}{(n,t)}\right)}{N}$, so we obtain
\[ T_{(a(1,\frac{t_n r}{(n,t)}) + b(0, \frac{rn}{(n,t)}))}(\tau + 1) = e\left(\frac{a\left(\frac{at_n r}{(n,t)} + \frac{brn}{(n,t)}\right)}{N} - \frac{c}{24}\right) T_{(a(1,\frac{t_n r}{(n,t)}) + b(0, \frac{rn}{(n,t)}))}(\tau). \]
This proves the third equation.

For the fourth equation, we note that by Theorem \ref{thm:vEMS}, the $S$-matrix entries $S_{(i,j),(k,l)}$ for ${}^gV$ are given by $\frac{1}{n} e(b_\Omega((i,j),(k,l)))$, where $b_\Omega$ is $-1$ times the inner product induced by $q_D$.  Since the map $B' \to D$ preserves the quadratic form, hence $-1$ times the $\bQ/\bZ$-valued inner product, if we write $b'_\Omega$ for the pullback of $b_\Omega$, we have
\[ \begin{aligned}
T_{(a(1,\frac{t_n r}{(n,t)}) + b(0, \frac{rn}{(n,t)}))}(\frac{-1}{\tau}) &= T_{(a_n,b_n+2t_n\lfloor a/n \rfloor)}(\unit, \frac{-1}{\tau}) \\
&= \frac{1}{n} \sum_{x \in D} e(b_\Omega((a_n,b_n+2t_n\lfloor a/n \rfloor),x)) T_x(\unit,\tau) \\
&= \frac{1}{N} \sum_{y \in B'} e(b'_\Omega((a,\frac{at_n r + brn}{(n,t)}),y)) T_y(\tau) \\
&= \frac{1}{N} \sum_{p,q \in \bZ/n\bZ} e\left(-\frac{\frac{at_n r + brn}{(n,t)}p+aq}{N}\right) T_{(p,q)}(\tau),
\end{aligned}\]
so the fourth equation holds for $(a,b) \in B'$.  For $(a,b) \notin B'$, we need a vanishing result, and it suffices to show that $\sum_{p,q \in \bZ/n\bZ} e(-\frac{bp+aq}{N}) T_{(p,q)}(\tau) = 0$.  This sum reduces to a sum over $B'$, and it suffices to show that for any $x \in D$, the sum over preimages of $x$ in $B'$ vanishes.  The kernel of $B' \to D$ is generated by $(n,\frac{-nrt_n}{(n,t)}) = n(1,\frac{t_n r}{(n,t)}) -2t_n (0, \frac{rn}{(n,t)})$, so we fix $p(1,\frac{t_n r}{(n,t)}) + q(0, \frac{rn}{(n,t)}) \in B'$ and compute the coefficient on $T_{p(1,\frac{t_n r}{(n,t)}) + q(0, \frac{rn}{(n,t)})}(\unit,\tau)$:
\[ \begin{aligned}
\sum_{k=0}^{N/n-1} &e\left( \frac{b(p+kn) + a(\frac{pt_nr+qrn-knrt_n}{(n,t)})}{N} \right) \\
&= \sum_{k=0}^{N/n-1} e\left( \frac{bp(n,t)+bkn(n,t) + aprt_n + aqrn + aknrt_n}{N(n,t)} \right) \\
&= e\left( \frac{bp(n,t)+ aprt_n + aqrn}{N(n,t)} \right)  \sum_{k=0}^{N/n-1} e\left( k\frac{bn(n,t) + anrt_n}{N(n,t)} \right)
\end{aligned} \]
The last sum vanishes if and only if the numerator $bn(n,t) +anrt_n$ is nonzero modulo $N(n,t)$, and by fixing $a$, a short calculation shows that this is precisely when $(a,b) \notin B'$.
\end{proof}

Generalized Moonshine for cyclic groups was proved in \cite{DLM00} up to a scalar ambiguity.  We can eliminate the ambiguity by using a central extension, essentially given by unrolling the corresponding abelian intertwining algebra.

\begin{defn} \label{defn:Fg}
For any $g \in \bM$, and let $N$ be any positive integer such that the McKay-Thompson series $T_g(\tau) = \Tr(g q^{L_0-1}|V^\natural)$ is $\Gamma_0(N)$-invariant. By \cite{GM2} Corollary 3.25, we may define the vector-valued function $\hat{F}^g = \{\hat{F}^g_{i,j}\}_{i,j=0}^{N-1}$ by $\hat{F}^g_{i,j} = T_{g^{(i,j,N)}}(\tau)|_0 A$ for $A = \left(\begin{smallmatrix} * & * \\ \frac{i}{(i,j,N)} & \frac{j}{(i,j,N)} \end{smallmatrix} \right) \in SL_2(\bZ/\frac{N}{(i,j,N)}\bZ)$.  Similarly, we may define the vector valued function $F^g_{i,j} = \{F^g_{i,j}\}_{i,j=0}^{N-1}$ by the discrete Fourier transform $F^g_{i,j}(\tau) = \frac{1}{N} \sum_{k \in \bZ/N\bZ} e\left(\frac{-jk}{N} \right) \hat{F}^g_{i,k}(\tau)$.  This is a vector-valued modular function for the Weil representation $\rho_{I\!I_{1,1}(1/N)}$.  We write $c^g_{i,j}(n)$ for the $q^n$ coefficient of $F^g_{i,j}$.
\end{defn}

\begin{cor}
For any $g \in \bM$, hypothesis $B_g$ from section 4.4 of \cite{GM2} is satisfied, i.e., if $T_g(\tau)$ has level $N$, then there are actions $\{ \psi_i \}_{i=0}^{N-1}$ of $\bZ/N\bZ$ on the twisted modules $V^\natural(g^i)$, compatible with the action of $\langle g \rangle$ on $V^\natural$, such that $\Tr(\psi^i(j) q^{L_0-1} | \dim V^\natural(g^i)) = \hat{F}^g_{i,j}(\tau)$, and $\psi_i(i) = e(L_0)$.
\end{cor}
\begin{proof}
The actions $\psi_i$ given in Proposition \ref{prop:unrolled-algebra} satisfy these identities, and $V^\natural$ satisfies the hypotheses of that proposition by Theorem \ref{thm:Monster-voa-is-well-behaved}.
\end{proof}

\subsection{Adding a spacetime torus}

In order to produce a Lie algebra from a quantization functor, we need a conformal vertex algebra of central charge 26.  Our abelian intertwining algebras ${}^g V$ and ${}^g_N V$ have central charge 24, so we have to increase the central charge by two and cancel the braiding.  Our functors amount to taking a graded tensor product with the abelian intertwining algebra attached to a well-chosen 2-dimensional hyperbolic lattice.

\begin{defn}
Let $L$ be an integer lattice of signature $(1,1)$, and let $(D,q)$ be the discriminant form $L^\vee/L$.  We define the ``add a torus'' functor $AT_L$ from $AIA^c_{(D,q)}$ to the category of $L$-graded conformal vertex algebras of central charge $c+2$ as follows:
\begin{enumerate}
\item For any object $W = \bigoplus_{a \in D} W_a$ in $AIA^c_{(D,q)}$ with abelian 3-cocycle $(F,\Omega)$, we choose an abelian 3-coboundary to endow $V_{L^\vee(-1)}$ with the abelian 3-cocycle $(\frac{1}{F}, \frac{1}{\Omega})$, and take the graded tensor product $W \otimes^D V_{L^\vee(-1)} = \bigoplus_{a \in D} W_a \otimes V_{L_{-1} + a}$.
\item For any isomorphism $W \to W'$ in $AIA^c_{(D,q)}$, the corresponding isomorphism of $L^\vee$-graded conformal vertex algebras is given by taking the tensor product.
\end{enumerate}
For any $\lambda \in L^\vee$, the $\lambda$-graded piece is the tensor product $W_{\lambda+L} \otimes \pi_\lambda^{L \otimes \bC}$.
\end{defn}

\begin{lem}
Let $V$ be a holomorphic $C_2$-cofinite vertex operator algebra of central charge $c$, let $g$ be an automorphism of order $n$ such that $V(g^i)$ has strictly positive $L_0$-spectrum for $g^i \neq 1$, and let ${}^gV$ be the abelian intertwining algebra given in Theorem \ref{thm:vEMS}.  Then there is a lattice $L$ of signature $(1,1)$ and an abelian intertwining algebra $V_{L^\vee(-1)}$ whose abelian 3-cocycle cancels that of ${}^gV$.
In particular, adding a torus is well-defined on ${}^gV$, and yields an $L^\vee(-1)$-graded conformal vertex algebra of central charge $c+2$.
\end{lem}
\begin{proof}
Let $L$ be the lattice identified in \cite{vEMS} by its Gram matrix $\left( \begin{smallmatrix} -2t_n & n \\ n & 0 \end{smallmatrix} \right)$.  Then the quadratic form on $L^\vee/L$ is precisely $q_\Omega$ for ${}^g V$.  Thus, there exists a suitable coboundary for $V_{L^\vee(-1)}$ such that the resulting abelian 3-cocycle cancels that of ${}^gV$.
\end{proof}

\begin{prop} \label{prop:adding-a-torus-commutes-with-unrolling}
Adding a torus 2-commutes with both unrolling against hyperbolic lattice embeddings and twisting against abelian 3-coboundaries.  That is, different choices of unrolling or coboundaries on an abelian intertwining algebra yield isomorphic conformal vertex algebras.
\end{prop}
\begin{proof}
The claim about coboundaries follows from the fact that if one twists against a coboundary for the object in $AIA^c_{(D,q)}$, the opposite coboundary is chosen on $V_{L^\vee(-1)}$.  Thus, the graded tensor product is unchanged.

Suppose we are given an embedding $L \subset M$ of integral lattices of signature $(1,1)$, and let $(A,q_A) = M^\vee/M$ and $(B,q_B) = L^\vee/L$.  We have the diagram $M^\vee/M \twoheadleftarrow M^\vee/L \hookrightarrow L^\vee/L$ of morphisms of quadratic groups, so the unrolling functor yields an abelian intertwining algebra on $L^\vee/L$ that is supported on $M^\vee/L$.  Adding a torus to the $A$-graded object yields an $M^\vee(-1)$-graded conformal vertex algebra, while adding a torus to the $B$-graded object yields an $L^\vee(-1)$-graded conformal vertex algebra that is supported on the sublattice $M^\vee(-1)$.  Furthermore, the homogeneous components are isomorphic, and the multiplication laws match, since the multiplication law on the $L^\vee(-1)$-graded vertex algebra is given by a tensor product of a lattice algebra with a multiplication rule pulled back from the $A$-graded abelian intertwining algebra.
\end{proof}

\begin{defn}
For any $g \in \bM$, we write ${}^gV^\natural$ for the abelian intertwining algebra in $AIA^{24}_{q_D}$ given by Theorem \ref{thm:vEMS} and Theorem \ref{thm:Monster-voa-is-well-behaved}.  We write ${}^g_N V^\natural$ for the ``unrolled'' abelian intertwining algebra in $AIA^{24}_{q_B}$ given by Proposition \ref{prop:unrolled-algebra}.
\end{defn}

Recall that by forgetting some multiplicative structure, we may view ${}^g V^\natural$ as a direct sum of irreducible $g^i$-twisted $V^\natural$-modules, as $g^i$ ranges over the cyclic group generated by $g$.  We note that similarly, ${}^g_N V^\natural$ can be viewed as $\bigoplus_{i=0}^{N-1} V^\natural(g^i)$.

\begin{cor} \label{cor:adding-a-torus-to-Monster}
For any $g \in \bM$, and any multiple $N$ of the level of $T_g(\tau)$, any choice of lattice embedding $I\!I_{1,1}(N) \hookrightarrow L$ (for $L^\vee/L$ the discriminant form for ${}^gV^\natural$) induces an isomorphism $AT_L({}^gV^\natural) \cong AT_{I\!I_{1,1}(N)} ({}^g_NV^\natural)$ of $I\!I_{1,1}(-1/N)$-graded conformal vertex algebras of central charge 26.
\end{cor}
\begin{proof}
Theorem \ref{thm:Monster-voa-is-well-behaved} implies Proposition \ref{prop:adding-a-torus-commutes-with-unrolling} is applicable, and the result immediately follows.
\end{proof}

\subsection{Projective action of centralizer}

We show that for any finite order automorphism $g$ of a holomorphic $C_2$-cofinite vertex operator algebra $V$ of CFT type, there is a natural action by a central extension of the centralizer of $g$ by homogeneous automorphisms on the abelian intertwining algebra ${}^gV$, the unrolled object ${}^g_N W$, and the conformal vertex algebra given by adding a torus.

\begin{lem} \label{lem:enumerating-lifts}
Let $V$ be a holomorphic $C_2$-cofinite vertex operator algebra of CFT type, and let $g$ be an automorphism of order $n$ such that $V(g^i)$ has strictly positive $L_0$-spectrum for $g^i \neq 1$.  Let $h \in C_{\Aut V}(g)$.  Then the set of homogeneous automorphisms of the abelian intertwining algebra ${}^gV$ that restrict to $h$ on $V$ has precisely $n$ elements, and a natural simply transitive action of $\mu_n(\bC)$.
\end{lem}
\begin{proof}
If $\tilde{h}$ is a homogeneous automorphism of the abelian intertwining algebra ${}^gV$ that restricts to $h$ on $V$, then for each $g^i \in \langle g \rangle$, the restriction of $\tilde{h}$ to $V(g^i)$ is an element of $\Aut V(g^i)$ over $h$.  Now, suppose we choose arbitrary elements $h_i \in \Aut V(g^i)$ over $h$ for all $i$.  By Remark \ref{rem:automorphisms-of-twisted-modules}, the set of lifts of $h$ to an automorphism of $V(g^i)$ has a simply transitive action of $\bC^\times$, so we can measure the failure of $\{h_i\}_{i=0}^{n-1}$ to form an automorphism of ${}^gV$ by setting $\lambda_{i,j} \in \bC^\times$ to be the unique constant such that $Y(h_i a,z)h_j b = \lambda_{i,j} h_{i+j} Y(a,z)b$ for all $a \in V(g^i)$ and $b \in V(g^j)$.  By the associativity and commutativity of fusion, $\lambda = \{ \lambda_{i,j}\}$ is a normalized symmetric 2-cocycle on $\bZ/n\bZ$ with values in $\bC^\times$, hence it defines an abelian central extension $E$ of $\bZ/n\bZ$ by $\bC^\times$.  Because $\bC^\times$ is divisible, the defining exact sequence splits, and $E \cong \bZ/n\bZ \oplus \bC^\times$.  In particular, the cocycle $\lambda$ is a coboundary, and the set of 1-cochains in the preimage of $\lambda$ under the differential is a torsor under the group $\Hom(\bZ/n\bZ, \bC^\times) \cong \mu_n(\bC)$ of 1-cocycles.  Concretely, this means any normalized symmetric 2-cocycle $\lambda$ is attainable by a suitable choice of $\{h_i\}$, and those $\{h_i\}$ that yield an automorphism of ${}^gV$ (namely those for which $\lambda = 1$) form a torsor under $\mu_n(\bC)$, where the action is by rescaling: For each $\zeta \in \mu_n(\bC)$, one replaces $h_i$ with $\zeta^i h_i$.
\end{proof}

\begin{prop} \label{prop:central-extension}
Let $V$ be a holomorphic $C_2$-cofinite vertex operator algebra of CFT type, and let $g$ be an automorphism of order $n$ such that $V(g^i)$ has strictly positive $L_0$-spectrum for $g^i \neq 1$.  Then there is a central extension $\widetilde{C_{\Aut V}(g)}$ of $C_{\Aut V}(g)$ by $\mu_n(\bC)$ with a natural action on the abelian intertwining algebra ${}^gV$, the unrolled object ${}^g_N W$, and the conformal vertex algebra given by adding a torus.
\end{prop}
\begin{proof}
By Lemma \ref{lem:enumerating-lifts}, the set of all lifts of elements of $C_{\Aut V}(g)$ is a $\mu_n(\bC)$-torsor over $C_{\Aut V}(g)$, so in order to prove it is a central extension, we need to show that it is closed under multiplication, the torsor structure map is a homomorphism, and all lifts of identity are central.  All lifts of identity are central, because the restriction of a lift of identity to any $V(g^i)$ is a scalar, and all lifts of elements are sums of linear automorphisms of the twisted modules $V(g^i)$.  The remaining claims about the multiplicative structure follow from the following calculation: if $\tilde{h} = \{h_i \}$ and $\tilde{k} = \{k_i \}$ are lifts of $h, k \in C_{\Aut V}(g)$, then for all $a \in V(g^i), b \in V(g^j)$,
\[ \begin{aligned}
h_{i+j} k_{i+j} Y(a,z)b &= h_{i+j} Y(k_i a,z) k_j b \\
&= Y(h_i k_i u,a) h_j k_j vb
\end{aligned} \]
Therefore, $\tilde{h}\tilde{k}$ is a lift of $hk$ to ${}^gV$.

We thus obtain a natural action of $\widetilde{C_{\Aut V}(g)}$ on ${}^gV$.  Since unrolling and adding a torus are functors, we get natural actions of the same group on the resulting objects.
\end{proof}

\begin{lem}
For any $N$ divisible by the level of $T_g(\tau)$, the $I\!I_{1,1}(1/N)$-homogeneous conformal vertex algebra isomorphism between $AT_L({}^g V^\natural)$ and $AT_{I\!I_{1,1}(N)}({}^g_N V^\natural)$ given in Corollary \ref{cor:adding-a-torus-to-Monster} can be made $\widetilde{C_{\bM}(g)}$-equivariant.
\end{lem}
\begin{proof}
Promotion to $\widetilde{C_{\bM}(g)}$-equivariance follows from the structure of the unrolling functor - the corresponding graded pieces are identified as $\widetilde{C_{\bM}(g)}$-modules, and the multiplication maps are equivariant.
\end{proof}

We summarize the results of this section.

\begin{thm} \label{thm:properties-we-need}
Let $g$ be an automorphism of $V^\natural$, i.e., an element of the Monster simple group $\bM$, and suppose the McKay-Thompson series $T_g(\tau)$ has level $N$.  Let $B = \bZ/N\bZ \times \bZ/N\bZ$ with the quadratic form $q_B(a,b) = e(ab/N)$.  Then there is a central extension $\widetilde{C_{\bM}(g)}$ of the centralizer $C_{\bM}(g)$ by $\mu_{|g|}(\bC)$ equipped with a distinguished lift $\tilde{g}$ of $g$, and an abelian intertwining algebra ${}^g_NV^\natural \in AIA^{24}_{q_B}$ with a homogeneous action of $\widetilde{C_{\bM}(g)}$ by autmorphisms, satisfying the following properties:
\begin{enumerate}
\item ${}^g_NV^\natural$ decomposes into $\bigoplus_{a=0}^{N-1} V^\natural(g^a)$ as a direct sum of twisted $V^\natural$-modules.
\item The action of $\tilde{g}$ restricted to $V^\natural(g)$ is given by $e(L_0)$.
\item The $(a,b)$-graded component of ${}^g_NV^\natural$ is $\{ v \in V^\natural(g^a) | \tilde{g}v = e(b/N)v \}$.  This is nonzero if and only if $(a,b) \in B'$ (where $B' \subseteq B$ is a subgroup uniquely determined by the symmetries of $T_g$), and in that case it is a simple current as a $(V^\natural)^g$-module, and it is endowed with a $\widetilde{C_{\bM}(g)}$-action.
\item The graded dimension of the $(a,b)$-graded component ${}^g_NV^\natural$ is given by the component form $F^g_{a,b}(\tau)$ given in Definition \ref{defn:Fg}, and in particular, the coefficient $c^g_{a,b}(n)$ is equal to the dimension of $\{ v \in V^\natural(g^a) | \tilde{g}v = e(b/N)v, L_0 v = (n+1)v\}$.  That is, the vector-valued graded dimension of ${}^g_NV^\natural$ is equal to $F^g$, a vector valued modular function for the Weil representation attached to the lattice $I\!I_{1,1}(N)$.
\end{enumerate}
The conformal vertex algebra $AT_{I\!I_{1,1}(N)}({}^g_N V^\natural)$ has central charge 26, and is graded by $I\!I_{1,1}(-1/N)$, with an action of $\widetilde{C_{\bM}(g)}$ by homogeneous automorphisms.
\end{thm}
\begin{proof}
As before, Theorem \ref{thm:Monster-voa-is-well-behaved} allows us to use the general results for a holomorphic $C_2$-cofinite vertex operator algebra $V$ of CFT type equipped with an automorphism $g$ for which $V(g^i)$ has strictly positive $L_0$-spectrum whenever $g^i \neq 1$.

By Lemma \ref{lem:enumerating-lifts}, a lift of $g$ to an automorphism of ${}^g_NV^\natural$ is uniquely determined by what it does to $V^\natural(g)$, and by the $\tau \mapsto \tau+1$ property given in Proposition \ref{prop:unrolled-algebra}, setting $\psi_1 = e(L_0)$ yields a lift $\tilde{g}$ such that the $(a,b)$-graded component of ${}^g_NV^\natural$ is $\{ v \in V^\natural(g^a) | \tilde{g}v = e(b/N)v \}$.  The claim that the $(a,b)$-component is nonzero if and only if $(a,b) \in B'$ is straightforward from the definition of unrolling, and the simple current assertion is given by the first listed claim in Proposition \ref{prop:unrolled-algebra}.

The claims about graded dimension follow from the properties of the functions $T(a,b;\tau)$ and $T_{(a,b)}(\tau)$ given in Proposition \ref{prop:unrolled-algebra}.  In particular, \cite{GM2} Lemma 3.6 implies the functions $\{T(a,b,\tau) \}_{a,b \in \bZ/N\bZ}$ and $\{F^g_{a,b}(\tau) \}_{a,b \in \bZ/N\bZ}$ on $\cM_{Ell}^{\bZ/N\bZ}$ are uniquely determined by $T(0,m,\tau)$ and $F^g_{0,m}(\tau)$ for $m$ ranging over divisors of $N$, and in both cases, these are given by the McKay-Thompson series $T_{g^m}(\tau)$.  The discrete Fourier transform then gives an equality between $T_{(a,b)}(\tau)$ and $F^g_{a,b}(\tau)$, matching the coefficient $c^g_{a,b}(n)$ with the dimension of the stated eigenspace.

The remaining claims follow from Proposition \ref{prop:central-extension} and Corollary \ref{cor:adding-a-torus-to-Monster}.
\end{proof}

\section{Lie algebras}

In this section we construct the monstrous Lie algebras.  Our construction employs the physical principle first articulated in \cite{L71} and demonstrated in \cite{GT72} and \cite{B72}, that in the critical dimension $d=26$ of bosonic string theory, Virasoro-type Ward identities cancel two full sets of oscillators.  This principle now has at least three methods of realization in physics (cf. \cite{GSW87} chapters 2,3): Covariant quantization, passage to the light-cone gauge (also called ``construction of the transversal space''), and BRST cohomology.  The use of the BRST method for quantizing strings was pioneered in \cite{KO83}.

The use of this cancellation principle in mathematics is also far from new.  For example, BRST cohomology was used in \cite{FGZ86} to compute the character of the space of physical states for a bosonic string propagating in Lorentzian 26-space.  The covariant quantization method is also used in the analysis of three Lie algebras, each of which was named the ``Monster Lie algebra'' upon construction:
\begin{enumerate}
\item The first example is the Kac-Moody Lie algebra whose simple roots are those of the reflection group of the even unimodular Lorentzian lattice $I\!I_{1,25}$.  The existence of a Monster action was suggested in \cite{BCQS84}, and Frenkel \cite{F85} used the no-ghost theorem to bound its root multiplicities by coefficients of the modular form $1/\Delta$.
\item The second example is the Borcherds-Kac-Moody Lie algebra whose real simple roots are those given in the first example, but it has additional imaginary simple roots of multiplicity 24 on positive multiples of the Weyl vector.  This was constructed in \cite{B86}, and is now known as the ``fake Monster Lie algebra''.  In \cite{B90}, Borcherds showed that the root multiplicities saturate the $1/\Delta$ bound given by Frenkel.
\item The third and last example is the generalized Kac-Moody Lie algebra constructed in \cite{B92}.  This is the only example known to have a faithful Monster action.
\end{enumerate}
Supersymmetric variants of this construction have been used, e.g., in \cite{S00}.

\subsection{Virasoro preliminaries}

Recall the Virasoro algebra $vir$ is the complex vector space with a distinguished basis given by symbols $K, \{L_{n} \}_{n \in \bZ}$, with the Lie bracket defined by $[L_{m},L_{n}] = (m-n)L_{m+n} + \frac{m^3-m}{12}\delta_{m+n,0} K$ for all $m,n \in \bZ$, and $[L_{n}, K] = 0$ for all $n \in \bZ$.  We write $vir^+$ for the subalgebra spanned by $\{L_{n}\}_{n \geq 1}$, $vir^-$ for the subalgebra spanned by $\{ L_{n} \}_{n \leq -1}$, and $vir^0$ for the 2-dimensional subalgebra spanned by $L_0$ and $K$.

\begin{defn}
Let $V$ be a representation of the Virasoro algebra, i.e., a complex vector space equipped with operators $\{ L_{n} \}_{n \in \bZ}$ and $K$ that satisfy the Virasoro relations.
\begin{enumerate}
\item We say that $V$ is positive-energy if:
\begin{enumerate}
\item $L_0$ acts diagonalizably, i.e., $V$ splits into a direct sum of $L_0$-eigenspaces.
\item $vir^+$ acts locally nilpotently, i.e., for all $v \in V$, there exists $N > 0$ such that for all sequences $n_1,\ldots,n_k$ of positive integers satisfying $n_1+\cdots + n_k > N$, we have $L_{n_1}\cdots L_{n_k}v = 0$.
\end{enumerate}
\item We say that $V$ has central charge $c \in \bC$ if $Kv = cv$ for all $v \in V$.
\item A bilinear or sesquilinear form $(,)$ on $V$ is Virasoro-invariant if for each $n \in \bZ$, $L_{n}$ is adjoint to $L_{-n}$.
\item We say that $V$ is unitarizable if it admits a positive-definite hermitian Virasoro-invariant form.
\item A vector $v \in V$ is:
\begin{enumerate}
\item primary (or physical) if $L_iv = 0$ for all $i > 0$, i.e., if $vir^+$ acts trivially.
\item spurious if it has the form $L_{-i} w$, for some $w \in V$ and some $i > 0$.
\item null if it is primary and spurious.
\end{enumerate}
We write $P^r$ to denote the subspace of primary vectors $v$ such that $L_0v = rv$.  We may write $P^r_V$ when there is more than one representation under consideration.  We write $null^r_V$ to denote the subspace of null vectors of weight $r$ in $V$.
\end{enumerate}
Given $c,h \in \bC$, we define the Verma module $M_{h,c}$ as the induced module $vir \otimes_{vir^+ \oplus vir^0} \bC_{h,c}$, where $\bC_{h,c}$ is the one dimensional vector space with trivial action of $vir^+$, such that $L_0$ has eigenvalue $h$, and $K$ has eigenvalue $c$.  We write $L_{h,c}$ for the unique irreducible quotient of $M_{h,c}$ with highest weight $h$.
\end{defn}

We note that any Verma module $M_{h,c}$ is positive-energy and admits a canonical invariant bilinear form, called the Shapovalov form, which is nonsingular if and only if $M_{h,c}$ is irreducible (see \cite{FF84} or \cite{KR87}).  In the following lemma, we set $c=24$, because that is the case we will use.  Similar results hold for other central charges in the open interval $(1,25) \subset \bR$.

\begin{lem} \label{lem:unitarizable}
Let $V$ be a positive energy $vir$-module of central charge $c=24$.  Then $V$ is unitarizable if and only if the $L_0$-spectrum of $V$ is non-negative and $L_{-1}$ annihilates the zero weight subspace $V_0 = \Ker L_0$.  If these conditions hold, then $V$ decomposes as an orthogonal direct sum of irreducible $vir$-modules as follows:
\[ V \cong \bigoplus_{h \geq 0} P^h_V \otimes L_{h,24} \cong (V_0 \otimes L_{0,24}) \oplus \bigoplus_{h >0} P^h_V \otimes M(c,h) \]
\end{lem}
\begin{proof}
The Kac determinant formula yields the following unitarizability and irreducibility results (where detailed proofs can be found in the first four sections of \cite{KR87}):
\begin{enumerate}
\item If $h \neq 0$, then $M_{h,24}$ is irreducible (i.e., equal to $L_{h,24}$), and the category of lowest weight $vir$-modules with central charge $24$ and lowest weight $h$ is completely reducible, with $L_{h,24}$ as its unique irreducible object.
\item For the case $h=0$, $M_{0,24}$ has a unique proper submodule $M_{1,24}$, and $L_{0,24}= M_{0,24}/M_{1,24}$.
\item $M_{h,24}$ is unitarizable if and only if $h > 0$.  $L_{h,24}$ is unitarizable if and only if $h \geq 0$.
\end{enumerate}
Given $V$ unitarizable, we choose a Virasoro-invariant Hermitian form.  By self-adjointness of $L_0$, any two $L_0$-eigenspaces with different eigenvalues are orthogonal.  For any coset $h+ \bZ$, we let $V_{h+\bZ}$ denote the subspace spanned by $L_0$-eigenvectors with eigenvalues in $h + \bZ$.  Then any disjoint cosets of $\bZ$ yield orthogonal subspaces of $V$, so we may consider one coset at a time.  Within each coset, the decomposition follows from induction by $L_0$-eigenvalue, noting that within each $L_0$-eigenspace the primary subspace is orthogonal to the spurious subspace.  The necessary and sufficient condition then follows immediately.
\end{proof}

\begin{lem} \label{lem:twisted-modules-unitarizable}
Let $g$ be an automorphism of $V^\natural$.  Then the irreducible $g$-twisted module $V^\natural(g)$ is unitarizable.  If $g \neq 1$, then $V^\natural(g)$ decomposes as a direct sum of Virasoro submodules of the form $M(c,h)$ as $h$ ranges over a multiset of positive rational numbers of the form $\frac{k}{|g|^2}$.
\end{lem}
\begin{proof}
For the case $g=1$, it suffices to show that $V^\natural$ is unitarizable - this is well-known, and follows straightforwardly from its conformal structure.  By Theorem \ref{thm:Monster-voa-is-well-behaved}, if $g \neq 1$, then the lowest weight space of $V^\natural(g)$ has positive $L_0$-eigenvalue, so by Lemma \ref{lem:unitarizable}, we may decompose $V^\natural(g)$ as a sum of unitarizable Verma modules.
\end{proof}

\begin{lem}
Let $\lambda \in \bR^{r,s}$, and let $\pi_\lambda^{r,s}$ denote the corresponding irreducible Heisenberg module, endowed with the Virasoro action arising from the conformal element $\frac12 \sum_\mu a^\mu_{-1} b^\mu_{-1} 1 \in \pi_0^{r,s}$.  Then $\pi_\lambda^{r,s}$ is a positive-energy $vir$-module of central charge 2 that admits a natural Virasoro-invariant Hermitian form for which $b_n^\mu$ is adjoint to $b_{-n}^\mu$, for all $n \in \bZ$ and all vectors $b^\mu \in \bR^{r,s}$.
\end{lem}
\begin{proof}
By irreducibility, the conditions on the Hermitian form determine it up to a normalization.  See, e.g., section 2.2 of \cite{FGZ86}.
\end{proof}

\begin{lem}
Let $V$ be a Virasoro representation equipped with a Virasoro-invariant bilinear form $(,)$.
\begin{enumerate}
\item The $L_0$-eigenspaces in $V$ are pairwise orthogonal with respect to the form.
\item $L_{-1}P^0 \subseteq P^1$.  In particular, $L_{-1}P^0 \subseteq null^1_V$.
\item A vector in $P^1$ is orthogonal to all other vectors in $P^1$ if and only if it lies in $null^1_V$.
\end{enumerate}
\end{lem}
\begin{proof}
All claims follow from straightforward calculations.
\end{proof}

There is a stronger notion of invariance for vertex algebras (found, e.g., in \cite{FHL93}).

\begin{defn}
Let $V$ be a conformal vertex algebra.  A bilinear form $(,)$ on $V$ is invariant if for all $u,v,w \in V$, the identity $(Y(u,z)v,w) = (v,Y(e^{zL_1}(-z^{-2})^{L_0}u, z^{-1})w)$ holds.  Equivalently, if $u$ has conformal weight $i$, then the adjoint of the operator $u_n$ is $(-1)^i \sum_{j \geq 0} (L_1^j(u))_{2i-j-n-2}/j!$.
\end{defn}

\begin{lem}
Let $V$ be a conformal vertex algebra.  Then any invariant bilinear form is Virasoro-invariant.
\end{lem}
\begin{proof}
This follows from evaluating the defining identity with $u$ chosen to be the Virasoro vector $\omega$.
\end{proof}

\begin{prop} \label{prop:duality-form}
Let $V$ be a holomorphic $C_2$-cofinite vertex operator algebra of CFT type, and let $g$ be a finite order automorphism of $V$ such that the twisted modules $V(g^i)$ have strictly positive $L_0$-spectrum for all $g^i \neq 1$.  Let $L$ be a hyperbolic lattice such that $L^\vee/L$ has the discriminant form of ${}^gV$, let $M \subset L$ be a hyperbolic sublattice of finite index, and consider the conformal vertex algebra $W$ given by applying $AT_M$ to the $M^\vee/M$-graded abelian intertwining algebra given by unrolling ${}^gV$ (by Proposition \ref{prop:adding-a-torus-commutes-with-unrolling}, this is isomorphic to $AT_L({}^gV)$ viewed as an $M^\vee(-1)$-graded vertex algebra supported on $L^\vee(-1)$).  Let $\langle, \rangle_W$ be the bilinear form on $W$ induced by the contragradient pairings between the component modules $V^{(i,j)}$ in ${}^gV$ and the contragradient pairings between the modules $\pi^{1,1}_\lambda$ and $\pi^{1,1}_{-\lambda}$ for the Heisenberg vertex algebra for $\bR^{1,1}$, normalized such that 
\[ \langle \unit, Y(e^{zL_1}(-z^2)^{-L_0} u, z^{-1})v \rangle_W = \langle Y(u,z)\unit, v \rangle_W \]
for all $u \in ({}^gV)^\lambda \otimes \pi^{1,1}_\lambda$ and $v \in ({}^gV)^{-\lambda} \otimes \pi^{1,1}_{-\lambda}$.  Then $\langle, \rangle_W$ is an invariant bilinear form on $W$ satisfying the following conditions:
\begin{enumerate}
\item For any $i,j \in L^\vee$ satisfying $i+j \neq 0$, we have $\langle u,v \rangle_W = 0$ for any $u \in ({}^gV)^i$ and $v \in ({}^gV)^j$.
\item If $L_0u = ru$ and $L_0v = sv$, and $r \neq s$, then $\langle u,v \rangle_W = 0$.
\item If $h \in \widetilde{C_{\Aut V}(g)}$, then $\langle hu,hv \rangle_W = \langle u,v \rangle_W$.
\end{enumerate}
\end{prop}
\begin{proof}
By Theorem 2.13 in \cite{S98} (which generalizes earlier results like \cite{L94} Theorem 3.1 to the setting we need), if $A$ is a conformal vertex algebra on which $L_1$ acts locally nilpotently, the dual space $\Hom(A_0/L_1 A_1, \bC)$ is isomorphic to the space of invariant bilinear forms on $A$, and an isomorphism is given by sending a map $f$ to the form $(u,v) = (f \circ \pi)(u_{-1}^*v)$, where
\begin{enumerate}
\item $\pi$ is the projection from $A$ to the weight zero space $A_0$.
\item $u_n^* = (-1)^i \sum_{j \geq 0} (\frac{L_1^j}{j!}u)_{2i-j-n-2}$ when $L_0u = iu$.
\end{enumerate}
We consider the unique functional $f$ on $W_0$ that sends the unit $\unit$ to $1$, and annihilates all weight zero $a \otimes b$ for $a \in \pi^{1,1}_\lambda$ and $b \in ({}^gV)^\lambda$ when $\lambda \in L^\vee(-1) \setminus \{ 0\}$.  Since $f$ annihilates all vectors in the image of $L_1$ on the weight one subspace, it corresponds to an invariant bilinear form $(,)_f$.  Since $({}^gV)^0 \otimes \pi^{1,1}_0$ has a unique invariant bilinear form up to normalization, the contragradient pairing coincides with $(,)_f$ on this subspace.  It is clear that both $\langle , \rangle_W$ and $(,)_f$ satisfy the homogeneity and $\widetilde{C_{\Aut V}(g)}$-invariance conditions listed in the statement of the proposition.  It remains to show that $(,)_f$ coincides with $\langle, \rangle_W$ in general, and for that it suffices to restrict our view to a pairing between $u \in ({}^gV)^\lambda \otimes \pi^{1,1}_\lambda$ and $v \in ({}^gV)^{-\lambda} \otimes \pi^{1,1}_{-\lambda}$.  In this case,
\[ \begin{aligned}
(u,v)_f &= (f \circ \pi)(u_{-1}^*v) \\
&= (\unit, u_{-1}^*v)_f \\
&= \langle \unit , u_{-1}^* v \rangle_W 
\end{aligned} \]
By our normalization, we have the equality
\[ \langle \unit, Y(e^{zL_1}(-z^2)^{-L_0} u, z^{-1})v \rangle_W = \langle Y(u,z)\unit, v \rangle_W \]
of elements of $\bC[z]$.  Isolating the constant term yields $\langle \unit , u_{-1}^* v \rangle_W = \langle u, v \rangle_W$.  We conclude that $(u,v)_f = \langle u, v \rangle_W$, so $\langle,\rangle_W$ is invariant.

\end{proof}

\subsection{Old covariant quantization}


In covariant quantization, one begins with a na\"ive space of states, that has redundancies due to internal symmetry.  In the case of internal Virasoro symmetry, to obtain the relevant space of states that couple, one considers the ``physical'' subspace of states that are invariant under the action of $vir^+$, and quotients by the ``spurious'' subspace, which is spanned by the image of the augmentation ideal in $U(vir^-)$ together with the radical of the induced Hermitian form.

An enhancement of the original covariant approach was used in \cite{B92} to construct the Monster Lie algebra as the space of physical states arising from the vertex algebra $V^\natural \otimes V_{I\!I_{1,1}(-1)}$.  We will use a variant of this method to construct our monstrous Lie algebras.

\begin{defn}
Let $\Rep^+_c(vir)$ denote the groupoid whose objects are positive energy Virasoro representations of central charge $c$ with Virasoro-invariant bilinear forms, and whose morphisms are Virasoro-equivariant isometries.  We define the functor $OCQ$, from $\Rep^+_c(vir)$ to the groupoid of vector spaces with bilinear forms, as $V \mapsto P^1_V/null^1_V$.
\end{defn}

\begin{lem} \label{lem:quantization-facts}
We list some relatively straightforward facts:
\begin{enumerate}
\item If $W$ is a conformal vertex algebra that is positive-energy, then $P^1/L_{-1}P^0$ has a Lie algebra structure induced by setting $[u,v] = u_0 v$, i.e., the coefficient of $z^{-1}$ in $Y(u,z)v$.  If a group $G$ acts on $W$ by conformal vertex algebra automorphisms, then $G$ acts on $P^1/L_{-1}P^0$ by Lie algebra automorphisms.
\item If $W$ is a positive-energy conformal vertex algebra equipped with an invariant bilinear form (in particular, with $L_i$ adjoint to $L_{-i}$ for each $i \in \bZ$), then the Lie algebra $P^1/L_{-1}P^0$ has an invariant inner product induced from $W$, and the radical of that inner product is an ideal.  In particular, $OCQ(W) = (P^1/L_{-1}P^0)/rad(,)$ is a Lie algebra equipped with a nondegenerate invariant inner product.  If a group $G$ acts on $W$ by conformal vertex algebra automorphisms, and the invariant bilinear form is $G$-invariant, then $OCQ(W)$ has a natural $G$ action by Lie algebra automorphisms, and the induced bilinear form is $G$-invariant.
\end{enumerate}
\end{lem}

\begin{proof}
The first sentence of the first claim is given in section 3 of \cite{B92}.  The Lie algebra structure of $W/L_{-1}W$ is known from \cite{B86}, so it suffices to show that $a_0 b$ is primary if $a$ and $b$ are.  This follows from the formula $L_n (a_0 b) = a_0 L_n b + \sum_{m \geq -1} \binom{n+1}{m+1} (L_m a)_{n-m} b$ for $n \geq -1$ together with the $L_{-1}$-derivative rule.  The second sentence follows from the fact that automorphisms preserve the Virasoro action (hence the property of being primary, the grading, and the $L_{-1}$ operator), and the Lie bracket $(u,v) \mapsto u_0v$.

The first sentence of the second claim is given in the proof of Theorem 6.1 of \cite{B92}.  Invariance follows immediately from the adjunction formula $u_n^* = (-1)^i \sum_{j \geq 0} (\frac{L_1^j}{j!}u)_{2i-j-n-2}$ when $n=0$, $i=1$, and $L_1 u=0$.  The ideal property follows immediately from invariance.  The part about group actions follows from the same reason as in the first claim.
\end{proof}

The following proposition seems to capture the core of Lovelace's principle of oscillator cancellation.  It is sometimes given the name ``No Ghost Theorem'', because it immediately implies that quantization yields a positive-definite space of states.  That is, if we choose a unitary structure on $V$, then transport of structure along the component isomorphisms yields a unitary structure on the output.  Here, the word ``ghost'' means negative-norm state, whose absence is important in physics, since the existence of such ghosts obstructs reasonable formulations of causality.  In the next section, we will consider a different type of ghost, named after Faddeev-Popov, and they get their name from their violation of the spin-statistics theorem.

\begin{prop} \label{prop:no-ghost}
Let $V$ be a unitarizable Virasoro representation of central charge 24 equipped with a non-degenerate Virasoro-invariant bilinear form, let $\pi^{1,1}_\alpha$ be the Heisenberg module attached to the linear function $\alpha$ on $\bR^{1,1}$, and let $W = V \otimes \pi^{1,1}_\alpha$.  Then:
\begin{enumerate}
\item If $\alpha \neq 0$, then $OCQ(W)$ is isomorphic to the subspace $V_{1-(\alpha, \alpha)}$ of $V$ on which $L_0$ acts by $1-(\alpha,\alpha)$, where $(,)$ is the induced bilinear form on the dual space of $\bR^{1,1}$.
\item If $\alpha=0$, then $OCQ(W)$ is isomorphic to $V_0 \otimes (\pi^{1,1}_0)_1 \oplus V_1 \otimes (\pi^{1,1}_0)_0 \cong V_0 \oplus V_0 \oplus V_1$, where $V_1$ is the subspace of $V$ on which $L_0$ acts by identity.
\end{enumerate}
\end{prop}
\begin{proof}
One may find a full proof of a slightly altered version of this claim in the appendix of \cite{J98}.  It is an expansion of the sketch of a proof for Theorem 5.1 in \cite{B92}, which is in turn a minor alteration of the argument in \cite{GT72}.  In order to pass from the proof in \cite{J98} to a proof of our claim, we need to make the following changes:
\begin{enumerate}
\item The vector $\alpha$ here is $r$ there, and we remove the restriction that it lie in the lattice $I\!I_{1,1}$, since that restriction is never used in the proof.
\item We remove the unnecessary assumption that $V$ is a vertex operator algebra, as the proof does not use any extra structure on the vector space $V$ except the Virasoro action.
\item The space $W$ here is $\mathcal{H}$ there.
\item Although it isn't necessary for this paper, we remove the assumption that the weight 0 subspace $V_0$ of $V$, defined as the kernel of $L_0$, is one dimensional.  One should therefore ignore the paragraph concerning the case $r =0$ at the end of the proof in \cite{J98}.  Instead, we note that all weight one elements of $V \otimes \pi^{1,1}_0$ are primary, so $P^1 \cong V_0 \oplus V_0 \oplus V_1$.  The fact that $L_{-1}$ annihilates $P^0$ follows from the classification of unitarizable Virasoro representations of central charge 24, given in Lemma \ref{lem:unitarizable}.  By our assumption on $V$ and the fact that $(\pi^{1,1}_0)_1$ has a natural real structure identified with $\bR^{1,1}$, the bilinear form on the quotient $P^1_W/L_{-1}P^0_W$ is nondegenerate, i.e., the radical vanishes.
\end{enumerate}
\end{proof}

\begin{defn} \label{defn:monstrous}
For each $g \in \bM$, we define:
\[ \fm_g = OCQ \circ AT_{I\!I_{1,1}(N)}({}^g_NV^\natural) \]
where $N$ is the level of $T_g(\tau)$.  This is a Lie algebra, by Lemma \ref{lem:quantization-facts}, and we call it the monstrous Lie algebra attached to $g$.
\end{defn}

\begin{rem}
By Proposition \ref{prop:adding-a-torus-commutes-with-unrolling}, $\fm_g$ is isomorphic to $OCQ \circ AT_L({}^gV^\natural)$ for any $L$ such that $L^\vee/L$ is the discriminant form for ${}^gV^\natural$.  Furthermore, we may replace $N$ with any multiple $rN$, and get an isomorphic object (although it is now a $I\!I_{1,1}(-1/rN)$-graded Lie algebra supported on the sublattice $I\!I_{1,1}(-1/N)$).
\end{rem}

\begin{prop} \label{prop:mg-properties}
Let $g \in \bM$.  Then the Lie algebra $\fm_g$ satisfies the following properties:
\begin{enumerate}
\item $\fm_g$ is graded by $I\!I_{1,1}(-1/N)$, and its degree (0,0) subspace is $\bC^2$, identified with the weight one space $(\pi^{1,1}_0)_1$ in the Heisenberg module.  Furthermore, the degree is given by the eigenvalue under the adjoint action of the degree $(0,0)$-subspace.
\item $\fm_g$ has a nondegenerate invariant bilinear form, and if $u$ and $v$ are homogeneous vectors of degree $a,b \in I\!I_{1,1}(-1/N)$ such that $a+b \neq 0$, then $u$ and $v$ are orthogonal.
\item $\fm_g$ has a canonical action of $\widetilde{C_{\bM}(g)}$ by homogeneous Lie algebra automorphisms that preserve the bilinear form.
\item For each nonzero $(a,b) \in I\!I_{1,1}(-1/N)$, the degree $(a,b)$ subspace is isomorphic as a $\widetilde{C_{\bM}(g)}$-module to the subspace of $V^\natural(g^a)$ on which $\tilde{g}$ acts by $e(\frac{b}{N})$ and $L_0$ acts by $1+\frac{ab}{N}$.  The dimension of this subspace is equal to the coefficient $c^g_{a,b}(ab/N)$ of the vector-valued modular function $F^g(\tau)$ given in Definition \ref{defn:Fg}.  In particular, we have a graded vector space isomorphism between $\fm_g$ and the Lie algebra $L_g$ defined in \cite{GM2} section 4.2 (written $W_g$ in Proposition 4.4 of that paper).
\end{enumerate}
\end{prop}
\begin{proof}
We prove the claims in the order they are stated.
\begin{enumerate}
\item The grading property follows from the fact that $AT_{I\!I_{1,1}(N)}({}^g_NV^\natural)$ is graded by $I\!I_{1,1}(-1/N)$.  Because $({}^gV^0)_1 =0$ and $({}^gV^0)_0 = \bC \unit$, Proposition \ref{prop:no-ghost} identifies the degree $(0,0)$ space with $V_0 \otimes (\pi^{1,1}_0)_1 = (\pi^{1,1}_0)_1$.  The eigenvalue assertion follows from the identification of the degree $(0,0)$ subalgebra, and the fact that for $u \in (\pi^{1,1}_0)_1$ and $v \in \pi^{1,1}_\alpha$, the $z^{-1}$-coefficient of $Y(u,z)v$ is $\alpha(u) v$.
\item The fact that the form is nondegenerate and invariant is given in Lemma \ref{lem:quantization-facts}.  The degree restriction follows from the corresponding property of $AT_{I\!I_{1,1}(N)}({}^g_NV^\natural)$, given in Proposition \ref{prop:duality-form}.
\item By Proposition \ref{prop:central-extension} (and Theorem \ref{thm:Monster-voa-is-well-behaved}), there is a canonical homogeneous action of $\widetilde{C_{\bM}(g)}$ on ${}^gV^\natural$ by homogeneous automorphisms, and this endows $AT_{I\!I_{1,1}(N)}({}^g_NV^\natural)$ with a corresponding action.  By the third claim of Lemma \ref{lem:quantization-facts}, homogeneous automorphisms of $AT_{I\!I_{1,1}(N)}({}^g_NV^\natural)$ are taken to homogeneous Lie algebra automorphisms of $\fm_g$.
\item The identification of $\widetilde{C_{\bM}(g)}$-modules follows from Proposition \ref{prop:no-ghost}.  
By the third listed claim of Theorem \ref{thm:properties-we-need}, the dimension of this space is equal to the Fourier coefficient $c^g_{a,b}(ab/N)$ of $F^g$.  To compare with the graded parts of $L_g$, we note that by \cite{GM2} Proposition 4.4, the $(a,b/N)$ root space of $L_g$ also has dimension equal to the Fourier coefficient $c^g_{a,b}(ab/N)$ of the vector-valued modular form $F^g$. 
\end{enumerate}
\end{proof}

\subsection{BRST cohomology}

In addition to covariant quantization, there is the ``modern'' method of Brecchi, Rouet, and Stora (with Tyutin's independent unpublished work often credited), where one extracts the relevant space of states as a cohomology group.  This has the advantage of being ``very functorial'', but the disadvantage of being rather complicated - one extracts the BRST current and the resulting differential by way of a super-variational principle.  In the string-theoretic context, BRST quantization was pioneered in \cite{KO83}.  One may find string-theoretic expositions of this method in chapter 3 of \cite{GSW87}, Chapter 4 of \cite{P98}, and section 7 of \cite{D97}, the last of which is intended for a more mathematical audience.

In the mathematical world, BRST quantization of Virasoro representations was introduced using semi-infinite cohomology in \cite{F84}, and generalized in \cite{FGZ86}.  Proofs can be found in \cite{LZ91}.  Borcherds's use of the no-ghost theorem for moonshine was reinterpreted in terms of BRST cohomology in \cite{LZ95}.  We will give a modified presentation that is applicable to the vertex algebra objects we have constructed in the previous section.


We introduce the ghost vertex superalgebra $V_{ghost}$ of central charge $-26$.  By boson-fermion correspondence, for any half-integer $\lambda$, there is an isomorphism between the bosonic vertex superalgebra attached to the odd positive definite lattice $\bZ$ with conformal vector shifted by $\lambda$, and the ghost $bc$-system of weight $\frac12 + \lambda$.  With the notation of \cite{FLM88}, the $bc$ system of weight $\frac12 + \lambda$ is generated by the ``anti-ghost'' field $b(z)$ that is identified with the bosonic lattice field $z^{-\lambda}X(1,z)$ of conformal weight $\frac12-\lambda$, and the ``ghost'' field $c(z)$ identified with $z^\lambda X(-1,z)$ of conformal weight $\frac12 + \lambda$.  The conformal vector is then given by $\frac12 h(-1)^2 + \lambda h(-2)$, and it has central charge $1-12\lambda^2$.  For $V_{ghost}$, we set $\lambda = \frac32$.

We briefly describe some cohomology functors and some key properties.  There is nothing original in this subsection except perhaps the method of presentation.
\begin{defn}
Given a Virasoro (``matter'') representation $V^m$ of central charge $D$ and energy-momentum tensor $T^m(z) = \sum_n L^m(n) z^{-n-2}$, we define the Virasoro module $C^{\infty/2 + *}(vir, V^m) = V^m \otimes V_{ghost}$, which has central charge $D-26$.  The old and new BRST currents are defined to be $j^{BRST,old} = cT^m + \frac12 :cT^{bc}:$ and $j^{BRST,new} = j^{BRST,old} + \frac{3}{2}\partial^2 c$.  The $U(1)_{bc}$ current is $j^{bc} = :bc:$.  We consider the following operations on $C^{\infty/2 + *}(vir, V^m)$:
\begin{enumerate}
\item The old and new BRST operators
\[ \begin{aligned}
Q^{old} &= j^{BRST,old}_0 = \frac{1}{2\pi i} \oint dz j^{BRST,old}(z) \\
&= \sum_n L^m(n) c_{-n} + \sum_{m<n} (m-n) :b_{m+n} c_{-m}c_{-n}: \\
Q^{new} &= j^{BRST,new}_0 = \frac{1}{2\pi i} \oint dz j^{BRST,new}(z) \\
&= \sum_{r \in \bZ} (L_r^{(m)} - \delta_{r,0})c_{-r} - \frac12 \sum_{r,s \in \bZ} (r-s) :c_{-r} c_{-s} b_{r+s}:
\end{aligned} \]
\item The old and new ghost number operators
\[ \begin{aligned}
U^{old} &= j^{bc}_0 = \frac12 (c_0 b_0 - b_0 c_0) + \sum_{n=1}^\infty (c_{-n} b_n - b_{-n} c_n) \\
U^{new} &= U^{old} + \frac{3}{2} = 1 + c_0 b_0 + \sum_{n=1}^\infty (c_{-n}b_n - b_{-n} c_n)
\end{aligned} \]
An eigenvector of $U$ with eigenvalue $k$ is said to have ghost number $k$.
\end{enumerate}
When $Q^2 = 0$, we call $C^{\infty/2 + *}(vir, V^m)$ the ``BRST complex''.  We also define the ``relative subcomplex'' (\cite{Z89}, section 4) to be the subspace of $V \otimes V_{ghost}$ annihilated by $b_0$.
\end{defn}

\begin{rem}
We have ``old'' and ``new'' versions of $Q$ and $U$, to reflect the use in the literature.  Papers from the 1980s and early 1990s (e.g., \cite{KO83}, \cite{FGZ86}, \cite{GSW87}, \cite{LZ91}) use the old versions, while the sources starting from the mid-1990s (e.g., \cite{P98}, \cite{D97}, \cite{LZ95}) add an additional $\frac{3}{2}\partial^2 c$ to the BRST current and a constant $3/2$ to the ghost number.  This change impacts the output of various functors in the following ways:
\begin{enumerate}
\item The addition of the derivative term makes the new BRST current a ``$(1,0)$-form as a quantum operator''.  Concretely, this means the OPE with the total stress-energy tensor $T^m + T^{bc}$ has the form $T(z)j^{BRST}(0) \sim \frac{1}{z^2}j^{BRST}(0) + \frac{1}{z} \partial j^{BRST}(0)$.  In the absence of the derivative, one has a pole of order 3 as the leading term.  The extra simplicity makes many calculations easier, especially when quantizing in higher genus settings.  However, it is irrelevant for the case at hand, since we just want to compute the cohomology spaces.
\item When we work out the space of physical states, we find that it lies in ghost degree $-1/2$ in the old convention and $1$ in the new convention.  In particular, the space of interest was written $H^{-1/2}_{BRST}$ in older work, but is written $H^1_{BRST}$ now.  The older convention has the convenient property that Poincar\'e duality gives a symmetry around zero (see \cite{FGZ86} Theorem 1.6 and the remark before Theorem 2.5), while the newer convention is more manageable when considering Gerstenhaber or BV structures on cohomology (\cite{LZ93}).
\end{enumerate}

The calculations relevant to us are unchanged, because a total derivative integrates to zero when evaluating charges.  That is, the old and new BRST cohomology functors are naturally isomorphic.  To avoid unnecessary confusion, we will write $H^1_{BRST}$ even when referring to results that use the old convention.




\end{rem}

\begin{lem}
Let $V^m$ be a positive-energy representation of Virasoro (called the ``matter representation'' in the physics literature).  We have the following operator calculations in $V^m \otimes V_{ghost}$:
\begin{enumerate}
\item $Q^2 = 0$ if and only if $D = 26$.  That is, $C^{\infty/2+*}(V^m)$ is a complex with differential $Q$ if and only if $V^m$ has central charge 26.
\item $[U,Q] = Q$, so $Q$ increases ghost number by one.
\item $[U,L_0] = 0$, so ghost number and $L_0$-eigenvalue provide a bigrading.
\end{enumerate}
Now, suppose $V^m$ has central charge $D=26$, so $Q$ gives $V^m \otimes V^{ghost}$ the structure of a complex of vector spaces.
\begin{enumerate}
\item $[Q,L_0]=0$, so the BRST complex and cohomology are graded by $L_0$-eigenvalues.
\item $[Q,b_0]= L_0$.  This implies the relative subcomplex is in fact a subcomplex, and all BRST-closed states with nonzero $L_0$-eigenvalue are BRST-exact.  In particular, the cohomology is supported in weight zero.
\end{enumerate}
\end{lem}

\begin{proof}
These claims follow from standard OPE calculations.  One may find them enumerated in section 4 of \cite{Z89}.  


\end{proof}

\begin{lem}
If $V$ is a conformal vertex algebra of central charge 26, then $H^1_{BRST}(V)$ is a Lie algebra, and morphisms of conformal vertex algebras are taken to maps of Lie algebras under this functor.  In particular, if a group $G$ acts on $V$ by conformal vertex algebra automorphisms, then $G$ naturally acts on $H^1_{BRST}(V)$ by Lie algebra automorphisms.
\end{lem}
\begin{proof}
For the Lie algebra claim, see \cite{LZ93} Theorem 2.2.  While the authors do not make the explicit claim that morphisms of conformal vertex algebras are taken to Lie algebra homomorphisms, the statement follows from the explicit formulas in their proof together with the faithfulness of the ``underlying vector space'' functors.  The claim about group actions then follows from the functoriality.
\end{proof}

We now have the BRST version of oscillator cancellation.

\begin{prop}
Let $\alpha \in \bR^{1,1}$, and let $\pi^{1,1}_\alpha$ be the Heisenberg representation attached to $\alpha$.  If $V$ is a unitarizable Virasoro representation of central charge $24$, then $H^1_{BRST}(V \otimes \pi^{1,1}_\alpha)$ is isomorphic to the subspace $V_{1-\alpha^2}$ of $V$ on which $L_0$ acts with eigenvalue $1-\alpha^2$ (if $\alpha \neq 0$), and $V_0 \oplus V_0 \oplus V_1$ (if $\alpha = 0$).  In particular, any positive definite Hermitian form on $V$ induces a positive definite Hermitian form on $H^1_{BRST}(V \otimes \pi^{1,1}_\alpha)$.
\end{prop}
\begin{proof}
The case $\alpha \neq 0$ is covered in \cite{Z89} Theorem 4.9 under the following translation of notation:
\begin{enumerate}
\item $\pi^{1,1}_\alpha$ is written $\mathcal{H}(D,p)$, with $D=2$ and $p = \alpha$.
\item $V$ is written $\mathcal{K}$.
\item The BRST differential on $\mathcal{H}(D,p) \otimes \mathcal{K} \otimes V^{gh}$ is written $Q_{mod}$.
\item $H_{BRST}^{u+3/2}(V \otimes \pi^{1,1}_\alpha)$ is written $H^u_{Q_{mod}}(p)$ for the same values of $u$, and $p=\alpha$.
\item The character $Tr(q^{L_0}|V)$ of $\mathcal{K}$ is written $\chi(\mathcal{K},q) = tr q^{J_0}$.
\end{enumerate}
The conclusion of the theorem is that $H^{-1/2}_{Q_{mod}}(p) \cong H^0_{rel}(p, Q_{mod})$, and that $\dim H^0_{rel}(p, Q_{mod})$ is the coefficient of $q^{-(p,p)/2}$ in the shifted character $Tr(q^{L_0-1}|V)$.  (The statement given in the paper contains an erroneous division by $\Delta$ rather than $q$ - this can be checked by comparison with the free boson case in Theorem 3.4, or using Corollary 2.27 in \cite{LZ91}, which is proved in more detail.)  This completes the case $\alpha \neq 0$.

The case $\alpha=0$ follows from an explicit computation of cocycle representatives, following the example of Proposition 2.9 in \cite{FGZ86}: One obtains a basis of BRST cohomology given by $\{ x \otimes L'_{-1} \wedge L'_{-2} \wedge \cdots \}$ as $x$ ranges over a basis of $(\pi^{1,1}_0 \otimes V)_1 = ((\pi^{1,1}_0)_1 \otimes V_0) \oplus ((\pi^{1,1}_0)_0 \otimes V_1) \cong V_0 \oplus V_0 \oplus V_1$.

As we remarked shortly before the statement of Proposition \ref{prop:no-ghost}, the positive definite Hermitian form on $H^1_{BRST}(V \otimes \pi^{1,1}_\alpha)$ is inherited from $V$ by the oscillator cancellation.

\end{proof}

We conclude with a natural isomorphism between the two quantization functors.  Let $\unit \in V_{ghost}$ be the ghost vacuum vector - it is the unique (up to scaling) vector that is annihilated by $\{b_n\}_{n \geq 0}$ and $\{c_n \}_{n > 0}$, and it has ghost number one.

\begin{lem}
The map $W \hookrightarrow W \otimes V_{ghost}$ given by $w \mapsto w \otimes \unit$ induces a natural isomorphism $OCQ \Rightarrow H^1_{BRST}$, as the input $W$ ranges over $\Rep^+_{26}(vir)$.
\end{lem}
\begin{proof}
This is proved at the end of section 4.4 in \cite{P98}.  Partial arguments can be found in section 3.2.1 of \cite{GSW87} and section 4.6 of \cite{D97} (which is in lecture 7 part F on the IAS web site).

When $W$ is restricted to be a direct sum of representations of the form $V \otimes \pi^{r,1}_\lambda$ for $V$ unitarizable of central charge $25-r$ and $r \leq 25$, one may instead appeal to the no-ghost theorem together with complete reducibility in the $V$ part.  This is the argument of Lemma 6.7 in \cite{LZ95}.
\end{proof}

\subsection{Borcherds-Kac-Moody property}

We discuss a class of infinite dimensional Lie algebras defined by Borcherds in \cite{B88}, where he called them ``generalized Kac-Moody algebras''.  Borcherds-Kac-Moody algebras are quite similar to Kac-Moody Lie algebras, since they are defined by generators and relations encoded by a generalized Cartan matrix, and they admit representation-theoretic data like Weyl character formulas.  They differ in the sense that the conditions defining a suitable generalized Cartan matrix are weaker, and most importantly, simple roots are allowed to be imaginary.  We will show that the Monstrous Lie algebras defined earlier in this section are Borcherds-Kac-Moody algebras, using a minor variant of the characterization from \cite{B95}, Theorem 1.  

The definitions of Borcherds-Kac-Moody algebras (or generalized Kac-Moody algebras) are not uniform in the literature.  We will use the definition given in \cite{J96} and \cite{J04}, because that is the definition that is both general enough to account for the Lie algebras we need and for which there is a published proof of a homological description of the denominator formula.  We needed the homological description in order to produce twisted denominator identities in \cite{GM2}.

\begin{defn}
If $I$ is a countable index set, a matrix $A = (a_{i,j})_{i,j \in I}$ of real numbers is called a generalized Cartan matrix if it satisfies the following conditions:
\begin{enumerate}
\item $A$ is symmetrizable, i.e., there is a diagonal matrix $Q$ whose diagonal entries $Q_{i,i} = q_i$ are positive real numbers, such that $QA$ is symmetric.
\item $a_{i,j} < 0$ if $i \neq j$.
\item If $a_{i,i} > 0$, then $a_{i,i} = 2$ and for all $j \in I$, $a_{i,j} \in \bZ$.
\end{enumerate}
Given a generalized Cartan matrix $(a_{i,j})_{i,j \in I}$, its universal Borcherds-Kac-Moody algebra $\fg(A)$ is the Lie algebra with generators $\{ h_i, e_i, f_i\}_{i,j \in I}$, and the following relations:
\begin{enumerate}
\item $\mathfrak{sl}_2$ relations: $[h_i, e_k] = a_{i,k} e_k$, $[h_i, f_k] = -a_{i,k} f_k$, $[e_i,f_j] = \delta_{i,j} h_i$.
\item Serre relations: If $a_{i,i} > 0$, then $\ad(e_i)^{1-2a_{i,j}}(e_j) = \ad(f_i)^{1-2a_{i,j}}(f_j) = 0$.
\item Orthogonality: If $a_{i,j} = 0$, then $[e_i, e_j] = [f_i,f_j] = 0$.
\end{enumerate}
A Borcherds-Kac-Moody algebra is a Lie algebra of the form $(\fg(A)/C).D$, where $C$ is a central ideal, and $D$ is a commutative Lie algebra of outer derivations.
\end{defn}

From these starting data, one can build up a theory that closely mirrors that of finite dimensional semisimple Lie algebras, but with some complications.  The general theory is substantially developed with proofs in \cite{J96}.  

There are two main differences between this definition and others:
\begin{enumerate}
\item Other definitions may start with the symmetrized matrix $QA$ instead of $A$.  This means one often sees the condition that if $a_{i,i} > 0$, then for all $j \in I$, $2a_{i,j}/a_{i,i} \in \bZ$.
\item Other definitions may have central generators $h_{i,j}$ for $i \neq j$, with the relation $[e_i,f_j] = \delta_{i,j} h_i$ replaced with $[e_i,f_j] = h_{i,j}$.  As we will see in the next lemma, these necessarily vanish in the cases of interest to us.  That is, if we were to use the more general definitions, these generators would end up in the central ideal $C$.
\end{enumerate}

Most of the Borcherds-Kac-Moody algebras that have naturally appeared in mathematical practice have infinitely many simple roots, but a finite dimensional Cartan subalgebra, i.e., $I$ is infinite, but one chooses $C$ to have finite codimension in the span of $\{ h_i\}_{i \in I}$.  We give a method to recognize most of the interesting cases:

\begin{lem} \label{lem:gkm-conditions}
Any complex Lie algebra $\fg$ satisfying the following conditions is Borcherds-Kac-Moody:
\begin{enumerate}
\item $\fg$ admits a nonsingular invariant symmetric bilinear form, i.e., $(x,y) = (y,x)$, $(x,[y,z]) = ([x,y],z)$ for all $x,y,z \in \fg$, and if $(x,y) = 0$ for all $y \in \fg$, then $x=0$.
\item $\fg$ has a self-centralizing subalgebra $H$, called a Cartan subalgebra, such that $\fg$ is the sum of eigenspaces under the action of $H$, and the nonzero eigenvalues, called roots, have finite multiplicity.
\item $H$ has a totally real structure $H_{\bR} \subset H$ (i.e., such that $H = H_{\bR} \oplus \sqrt{-1}H_{\bR}$), such that the restriction of the bilinear form is real-valued, and all roots are elements of the dual $H_{\bR}^\vee$.
\item There exists an element $h \in H_{\bR}$, such that:
\begin{enumerate}
\item the centralizer of $h$ in $\fg$ is $H$,
\item for any $M \in \bR$, there exist at most finitely many roots $\alpha$ such that $|\alpha(h)| < M$.
\end{enumerate}
This vector $h$ is called a regular element.  If $\alpha(h) < 0$ then we say that $\alpha$ is negative, and if $\alpha(h) > 0$ then we say that $\alpha$ is positive.
\item The norms of roots under the inner product $(,)$ are bounded above.
\item Any two roots of non-positive norm that are both positive or both negative have inner product at most zero, and if the inner product is zero, then their root spaces commute.
\end{enumerate}
\end{lem}
\begin{proof}
This is essentially theorem 1 in \cite{B95}, but we remove the hypothesis that $\fg$ is defined over the real numbers.  Borcherds's proof still works, because the reconstruction of $\fg$ by recursively adding simple roots following increasing values of $\alpha(h)$ only uses the real structure on $H$, not on $\fg$.  To match Borcherds's definition with ours, we would need to allow the generators $h_{i,j}$, but Borcherds shows in the course of his proof that $h_{i,j} = [e_i, f_j] = 0$ when $i \neq j$.  Thus, the proof shows that we obtain a Borcherds-Kac-Moody algebra in our restricted sense.
\end{proof}

Recall that an element $g \in \bM$ is Fricke if and only if the McKay-Thompson series $T_g$ is invariant under the level $N$ Fricke involution $\tau \mapsto \frac{-1}{N\tau}$ for some $N$.

\begin{prop} \label{prop:mg-is-bkm}
For each Fricke element $g \in \bM$, the Lie algebra $\fm_g$ from Definition \ref{defn:monstrous} is a Borcherds-Kac-Moody Lie algebra.
\end{prop}
\begin{proof}
It suffices to check that the conditions in Lemma \ref{lem:gkm-conditions} hold.
\begin{enumerate}
\item The bilinear form is induced from the contragradient form, as shown in Proposition \ref{prop:duality-form}.
\item We define $H$ to be the $(0,0)$-graded piece of $\fm_g$.  By the first claim in Proposition \ref{prop:mg-properties}, this is isomorphic to $\bC^2$.  Furthermore, the left multiplication by $H$ is given by the Heisenberg action on the tensor product vertex algebra, and their eigenvalues are given by degree.  The fact that anything not in $H$ has nonzero eigenvalue implies the self-centralizing condition is satisfied.
\item $H$ is naturally identified with $I\!I_{1,1}(-1/N) \otimes_{\bZ} \bC$, which has a natural real structure $H_{\bR} = I\!I_{1,1}(-1/N) \otimes_{\bZ} \bR$.  The inner product on $H_{\bR}$ is real valued, and the roots, viewed as elements of the lattice, lie in $H_{\bR}^\vee$.
\item By identifying $H_{\bR}$ with $I\!I_{1,1}(-1/N) \otimes_{\bZ} \bR$, we let $h$ be any negative norm element whose inner product with any root is nonzero.  A standard choice is $(1,2)$, under the identification $I\!I_{1,1}(-1/N)  \cong \bZ \oplus \frac{1}{N}\bZ \subset \bR^{1,1}$. 
\item If we write the norm of a root at $r \in I\!I_{1,1}(-1/N)$ as $r^2$, then the fourth claim of Proposition \ref{prop:mg-properties} implies the root space at $r$ is identified with a subspace of ${}^g_NV^\natural$ whose vectors have $L_0$-eigenvalue equal to $1-r^2$.  The $L_0$-eigenvalues in ${}^g_NV^\natural$ are bounded below by zero, so the norms of roots are bounded above by $1$. 
\item Because $H$ has a real Lorentzian structure, we can identify the real span of the root space with $\bR^{1,1}$, and any pair of imaginary positive (resp. negative) roots lies in a single quadrant, where their inner products are bounded above by zero.  Two such roots are orthogonal if and only if they are positive multiples of the same norm zero vector.  We claim that because $g$ is Fricke, there are no positive roots of norm zero.  Such a root would have the form $(a,0)$ or $(0,b/N)$ in $I\!I_{1,1}(-1/N)$ for $a,b > 0$, and by the last claim in Proposition \ref{prop:mg-properties}, the multiplicities of those root spaces are given by the coefficients $c^g_{a,0}(0)$ and $c^g_{0,b}(0)$ of $F^g$.  By Corollary 3.25 of \cite{GM2}, these coefficients are in turn the exponents attached to $(1-p^a q^0)$ and $(1-p^0 q^{b/N})$ in the product expansion
\[ T_g(\sigma) - T_g(-1/\tau) = p^{-1} \prod_{m>0,n \in \frac{1}{N}\bZ} (1-p^m q^n)^{c^g_{m,Nn}(mn)}. \]
Because $g$ is Fricke, we may rewrite the left side as $T_g(\sigma) - T_g(\tau/N)$, and we find that the product expansion has the form $p^{-1}(1-pq^{-1/N})$ times a product for which $m$ and $n$ are strictly positive.  Thus, the coefficients vanish, and there are no roots of norm zero.
\end{enumerate}
\end{proof}

When $g$ is non-Fricke, the last condition is somewhat more complicated to verify due to the presence of norm zero simple roots.  Since the non-Fricke case is not needed in this paper, we postpone it to future work.

The following lemma will be useful in the next section, when we consider concrete quantitative questions.  It is well-known to experts, but we couldn't find a statement in the literature that was both precise and concise.

\begin{lem} \label{lem:bkm-unique-roots}
Let $\fg_1$ and $\fg_2$ be Borcherds-Kac-Moody Lie algebras with finite dimensional Cartan subalgebras $\fh_1, \fh_2$.  Given an inner-product-preserving isomorphism $f: \fh_1 \to \fh_2$, there exists an extension to a Lie algebra isomorphism $\fg_1 \to \fg_2$ if and only if the root multiplicities are identical under the isometry $f$.
\end{lem}
\begin{proof}
It is clear that an isomorphism of Lie algebras that restricts to an isometry of Cartan subalgebras preserves root multiplicities.

Now, assume $\fg_1$ and $\fg_2$ have identical root multiplicities under the isometry $\fh_1 \to \fh_2$.  It is well-known that the Weyl-Kac-Borcherds denominator formula may be used to uniquely determine the simple roots of a Borcherds-Kac-Moody algebra from the root multiplicities.  Indeed, the set of values of $\alpha(h)$ as $\alpha$ ranges over positive roots is a discrete subset of the positive reals (in particular, well-ordered), and this allows us to recursively match the truncated evaluations of both sides of the Weyl-Kac-Borcherds denominator formula by adding simple roots as necessary.  Indeed, this is essentially the same reconstruction process as the one sees in the proof of Lemma \ref{lem:gkm-conditions}.  Thus, under a choice of matching regular elements, $\fg_1$ and $\fg_2$ have the same simple roots, i.e., the simple roots that are identified under the isometry $\fh_1 \to \fh_2$ have equal multiplicity.  Abstractly, this implies that both $\fg_1$ and $\fg_2$ are isomorphic to the Borcherds-Kac-Moody Lie algebra given by generators and relations from the same generalized Cartan matrix.  More concretely, the simple root spaces are isomorphic vector spaces, and any choice of isomorphism between simple root spaces that are identified through the isometry $\fh_1 \to \fh_2$ extends uniquely to a Lie algebra isomorphism.
\end{proof}

\begin{rem}
Theorem 7.2 of \cite{B92} is essentially an application of the previous lemma to identify the Monster Lie algebra $M = \fm_1$ defined by applying the no-ghost theorem to $V^\natural$ with the abstract Lie algebra $N = L_1$ defined by a specification of simple roots.
\end{rem}

\section{Comparison theorems}

\subsection{Lie algebra comparison} \label{sect:comparisons}

In \cite{GM2} section 4, we defined for each $g \in \bM$ a Borcherds-Kac-Moody Lie algebra $L_g$ (written $W_g$ in Proposition 4.4 of that paper), whose Weyl-Kac-Borcherds denominator identity is an infinite product expansion of $T_g(\sigma) - T_g(-1/\tau)$.  We now show that the quantization functor yields an isomorphic Lie algebra.

\begin{prop} \label{prop:lg-isom-mg}
Let $g \in \bM$ be a Fricke element.  Then there exists a homogeneous Lie algebra isomorphism $\fm_g \to L_g$.
\end{prop}
\begin{proof}
By Proposition \ref{prop:mg-is-bkm}, $\fm_g$ is a Borcherds-Kac-Moody Lie algebra whose Cartan subalgebra has a real Lorentzian structure.  By Theorem 4.2 of \cite{GM2}, $L_g$ is a Borcherds-Kac-Moody Lie algebra whose Cartan subalgebra has a real Lorentzian structure.  By Lemma \ref{lem:bkm-unique-roots}, if we are given two Borcherds-Kac-Moody Lie algebras and an inner-product-preserving isomorphism between their Cartan subalgebras, this isomorphism can be extended to an isomorphism of the full Lie algebras if and only if the root multiplicities are equal.

If $T_g(\tau)$ has level $N$, then the root lattices of both Lie algebras are identified with $I\!I_{1,1}(-1/N)$, so it suffices to compare dimensions of the homogeneous spaces of nonzero degree.  The dimensions are equal by the last claim in Proposition \ref{prop:mg-properties}, which gives a $I\!I_{1,1}(-1/N)$-graded vector space isomorphism $\fm_g \to L_g$.  Thus, we have equality of dimensions of root spaces, and hence an isomorphism of Lie algebras.
\end{proof}

\begin{cor} \label{cor:hypothesis-cg}
For any Fricke $g \in \bM$, hypothesis $C_g$ from section 4.4 of \cite{GM2} is satisfied, i.e., there is an action of $\widetilde{C_{\bM}(g)}$ on $L_g$ by Lie algebra automorphisms, such that the action on the $(a,b/N)$ root space coincides with the action on the subspace of $V^\natural(g^a)$ on which $\tilde{g}$ acts by $e(b/N)$ and $L_0$ acts by $1+ab/N$.
\end{cor}
\begin{proof}
If we fix a Lie algebra isomorphism $\fm_g \to L_g$, then the action of $\widetilde{C_{\bM}(g)}$ on $L_g$ is given by \textit{transport de structure}.  The isomorphism between $\widetilde{C_{\bM}(g)}$-module structures for the root spaces and the subspaces of twisted modules then follows from the corresponding fact for $\fm_g$, which holds by the fourth listed claim of Proposition \ref{prop:mg-properties}.
\end{proof}

\subsection{Main theorem}

The following theorem is the target of this series of papers.

\begin{thm} \label{thm:hauptmodul-for-Fricke}
If $g$ is a Fricke element, then for any $h \in C_{\bM}(g)$, there exists a discrete group $\Gamma_{g,h} \subset PSL_2(\bR)$ containing a congruence subgroup $\Gamma(K)$ for some $K > 0$ (hence commensurable with $SL_2(\bZ)$), such that for any lift $\tilde{h}$ of $h$ to a linear transformation on $V(g)$, the power series $Tr(\tilde{h} q^{L_0-1} |V^\natural(g))$ is the $q$-expansion of a modular function on $\cH$ that is invariant under $\Gamma_{g,h}$, and generates the function field of the quotient $\cH/\Gamma_{g,h}$.
\end{thm}
\begin{proof}
By Corollary \ref{cor:hypothesis-cg}, hypothesis $C_g$ holds, and in particular the positive subalgebra $E_g$ of $L_g$ is Fricke-compatible for the data $(\widetilde{C_{\bM}(g)}, \tilde{g}, \{ V^{i,j}_k = V^\natural(g^i)^{\tilde{g}=e(j/N),L_0-1=k} \})$ in the sense of \cite{GM2} section 4.2.  That is:
\begin{enumerate}
\item $E_g$ is graded by $\bZ_{>0} \times \frac{1}{N} \bZ$, and decomposes as a sum $\bigoplus_{n>0, m \in \frac{1}{N} \bZ} E_{i,j} p^i q^j$ of finite dimensional homogeneous components, where $p$ and $q$ denote degree shifts by $(0,1)$ and $(1,0)$.
\item $E_g$ admits an action of $\widetilde{C_{\bM}(g)}$ by homogeneous Lie algebra automorphisms, such that $E_{i,j} \cong V^{i,j}_{1+ij}$ as $\widetilde{C_{\bM}(g)}$-modules.
\item The homology of $E_g$ is given by
\begin{itemize}
\item $H_0(E) = \bC$
\item $H_1(E) = \bigoplus_{n \in \frac{1}{N}\bZ} V^{1,n}_{1+n} pq^n$
\item $H_2(E) = p \bigoplus_{m=1}^\infty V^{1,-1/N}_{1-1/N} \otimes V^{m,1/N}_{1+m/N} p^m$
\item $H_i(E) = 0$ for $i>2$.
\end{itemize}
\item $E_{1,-1/N} \cong V^{1,-1/N}_{1-1/N}$ is one dimensional.
\end{enumerate}

By Proposition 6.3 in \cite{GM1}, if $E_g$ is Fricke-compatible for $(\widetilde{C_{\bM}(g)}, \tilde{g}, \{ V^{i,j}_k\})$, then for any $\tilde{h} \in \widetilde{C_{\bM}(g)}$, the power series $Tr(\tilde{h} q^{L_0-1}|V^\natural(g))$ is the $q$-expansion of a Hauptmodul whose invariance group contains a principal congruence subgroup.
\end{proof}

By combining our main theorem with previously known results, we obtain a positive resolution of Norton's conjecture.

\begin{thm} (Generalized Moonshine for twisted modules of $V^\natural$) \label{thm:main}
Any rule that assigns
\begin{itemize}
\item to each $g \in \bM$ the irreducible $g$-twisted $V^\natural$-module $V^\natural(g)$, with its canonical projective $C_{\bM}(g)$-action, and
\item to each commuting pair $(g,h)$ the function $Z(g,h;\tau) = \Tr(\tilde{h} q^{L_0-1}|V^\natural(g))$ for some lift $\tilde{h}$ of $h$ to a linear transformation on $V^\natural(g)$
\end{itemize}
satisfies the following conditions:
\begin{enumerate}
\item The formal series defining $Z(g,h;\tau)$ is the $q$-expansion of a holomorphic function on the upper half-plane $\fH$.
\item The function $(g,h) \mapsto Z(g,h;\tau)$ is invariant under simultaneous conjugation on the pair $(g,h)$ in $\bM$, up to rescaling.
\item $Z(g,h;\tau)$ is either a constant or a Hauptmodul.
\item For any $\left(\begin{smallmatrix} a & b \\ c & d \end{smallmatrix} \right) \in SL_2(\bZ)$ and any commuting pair $(g,h)$ in $\bM$, $Z(g,h,\frac{a\tau+b}{c\tau+d})$ is proportional to $Z(g^a h^c, g^b h^d,\tau)$.
\item $Z(g,h;\tau)$ is proportional to $J(\tau)$ if and only if $g=h=1$.
\end{enumerate}
\end{thm}
\begin{proof}
The first claim is well-known: As a statement about normal convergence of power series, it follows from Theorem 8.1 of \cite{DLM00}.

The second claim is also well-known, and follows by general nonsense from the Dong-Li-Mason proof of existence and uniqueness for irreducible twisted $V^\natural$-modules: For any $k \in \bM$, $k$-conjugation on twisted modules induces an equivalence from the category of $g$-twisted $V^\natural$-modules to the category of $kgk^{-1}$-twisted $V^\natural$-modules, taking the projective $C_{\bM}(g)$-action on $V^\natural(g)$ to the projective $C_{\bM}(kgk^{-1})$-action on $V^\natural(kgk^{-1})$ by conjugation.  This preserves characters up to rescaling.

The third claim is relatively straightforward in light of Theorem \ref{thm:hauptmodul-for-Fricke} combined with the $SL_2(\bZ)$ claim.  If $g^a h^c$ is Fricke for some $a,c$ satisfying $(a,c) = 1$, then $Z(g^a h^c, g^b h^d,\tau)$ is a Hauptmodul for all $b,d$ such that $\left(\begin{smallmatrix} a & b \\ c & d \end{smallmatrix} \right) \in SL_2(\bZ)$.  The $SL_2(\bZ)$ claim then implies $Z(g,h,\tau)$ is proportional to a M\"obius transform of $Z(g^a h^c, g^b h^d,\tau)$, hence is a Hauptmodul.  On the other hand, if all $g^a h^c$ for $(a,c) = 1$ are non-Fricke, then the twisted modules $V^\natural(g^a h^c)$ have $L_0$-spectrum bounded below by 1, so $Z(g^a h^c, g^b h^d,\tau)$ is regular at infinity for all $\left(\begin{smallmatrix} a & b \\ c & d \end{smallmatrix} \right) \in SL_2(\bZ)$.  Then the $SL_2(\bZ)$ claim implies $Z(g,h,\tau)$ is regular at all cusps, so the maximum modulus principle implies it is a constant function.

As we remarked in the introduction, the $SL_2(\bZ)$ claim follows immediately from combining Theorem 10.1 in \cite{DLM00} (which reduces the question to a $g$-rationality condition) with Corollary 5.26 of \cite{CM16} (which asserts $g$-rationality).

The fifth claim also follows from \cite{DLM00}.  In particular, the positivity of the $L_0$-spectrum of $V^\natural(g)$ for $g \neq 1$ (as given in Theorem \ref{thm:Monster-voa-is-well-behaved}) eliminates all nontrivial twistings.  For the trivial twisting, the fact that the McKay-Thompson series $T_h(\tau) = Z(1,h,\tau)$ is equal to $J$ if and only if $h=1$ follows from the faithfulness of the monster action on $V^\natural$ as shown in \cite{FLM88}.
\end{proof}

\end{document}